\documentclass[11pt]{article}

\usepackage{layout}
\usepackage[makeroom]{cancel}
\usepackage{bbm}
\usepackage[affil-it]{authblk}

\usepackage{amsfonts}
\usepackage{amsmath}
\usepackage{amssymb}
\usepackage{amsthm}
\usepackage{graphicx} \usepackage{enumerate} \usepackage{multicol}
\usepackage{mathrsfs} \usepackage[all,cmtip]{xy}
\usepackage{enumerate}
\usepackage{xcolor}

\usepackage{cite}
\usepackage[textwidth=3.2cm]{todonotes}

\usepackage[colorlinks,linkcolor=blue,citecolor=blue,pagebackref,hypertexnames=false, breaklinks]{hyperref}

\usepackage{float}

\newlength{\bibitemsep}
\newlength{\bibparskip}
\let\oldthebibliography\thebibliography
\renewcommand\thebibliography[1]{%
  \oldthebibliography{#1}%
  \setlength{\parskip}{\bibitemsep}%
  \setlength{\itemsep}{\bibparskip}%
}

\newtheorem{theorem}{Theorem}[section]

\newtheorem{definition}[theorem]{Definition}
\newtheorem{proposition}[theorem]{Proposition}
\newtheorem{prop}[theorem]{Proposition}
\newtheorem{corollary}[theorem]{Corollary}
\newtheorem{lemma}[theorem]{Lemma}
\newtheorem{remark}[theorem]{Remark}
\newtheorem{example}[theorem]{Example}
\newtheorem{examples}[theorem]{Examples}
\newtheorem{foo}[theorem]{Remarks}

\newcommand{\brak}[1]{\left(#1\right)} 
\newcommand{\norm}[1]{{\left\lVert{#1}\right\rVert}}

\def\vint{\mathop{\mathchoice%
          {\setbox0\hbox{$\displaystyle\intop$}\kern 0.22\wd0%
           \vcenter{\hrule width 0.6\wd0}\kern -0.82\wd0}%
          {\setbox0\hbox{$\textstyle\intop$}\kern 0.2\wd0%
           \vcenter{\hrule width 0.6\wd0}\kern -0.8\wd0}%
          {\setbox0\hbox{$\scriptstyle\intop$}\kern 0.2\wd0%
           \vcenter{\hrule width 0.6\wd0}\kern -0.8\wd0}%
          {\setbox0\hbox{$\scriptscriptstyle\intop$}\kern 0.2\wd0%
           \vcenter{\hrule width 0.6\wd0}\kern -0.8\wd0}}%
          \mathopen{}\int}

\newcommand{\R}{\mathbb R}

\newcommand{\DF}{\mathcal{E}}

\newcommand{\Ecal}{\mathcal E}

\newcommand{\B}{\mathbf B}

\newcommand{\ve}{\varepsilon}

\title{Besov class via heat semigroup on Dirichlet spaces I: Sobolev type inequalities}

\author{Patricia Alonso Ruiz, 
Fabrice Baudoin,
Li Chen,
Luke~G. Rogers, 
Nageswari Shanmugalingam,
Alexander Teplyaev}

\date{}

\begin{document}

\maketitle

\begin{abstract}
We introduce heat semigroup-based Besov classes in the general framework of Dirichlet spaces. General properties of those classes are studied and quantitative regularization estimates for the heat semigroup in this scale of spaces  are obtained. As a highlight of the paper, we obtain  a far reaching $L^p$-analogue, $p \ge 1$, of the   Sobolev inequality that was  proved for $p=2$ by N. Varopoulos under the assumption of  ultracontractivity  for the heat semigroup. The case $p=1$ is of special interest since it  yields isoperimetric type inequalities.
\end{abstract}

\tableofcontents

\section{Introduction}

\subsection*{Motivation}

The family of Sobolev inequalities plays a major role in analysis (see for instance \cite{saloff2002} and references therein). In the Euclidean space $\mathbb{R}^n$, $n \ge 2$, it reads
\begin{align}\label{sobo-intro}
\| f \|_{L^q (\mathbb{R}^n)} \le C \| \, |  \nabla f  | \, \|_{L^p (\mathbb{R}^n)}, \quad f \in C_0^\infty(\mathbb{R}^n)
\end{align}
where $1 \le p <n$, $q=\frac{np}{n-p}$, and the constant $C$ depends on $n$ and $p$. 

 In the context of a measure space $(X, \mu)$ equipped with a symmetric Dirichlet form  $\mathcal{E}$ with domain $\mathcal{F}$, N. Varopoulos proved in \cite{Var85} that a  heat kernel upper bound of the type $p_t(x,y)\le \frac{C}{t^{n/2}}$, $n >2 $,  implies  the following Sobolev inequality
 \begin{align}\label{sobo-intro2}
 \| f \|_{L^q (X,\mu)} \le C \sqrt{\mathcal{E}(f,f)}, \quad f \in \mathcal F,
 \end{align}
where $q=\frac{2n}{n-2}$. 

When applied to the Euclidean space equipped with the standard Dirichlet form $\mathcal{E}(f,f)=\| \, |  \nabla f  | \, \|^2_{L^2 (\mathbb{R}^n)}$, Varopoulos' theorem yields the case $p=2$ in \eqref{sobo-intro}. When $p \neq 2$,  it is  natural to look for results extending the inequalities \eqref{sobo-intro} to Dirichlet spaces.  In the present paper, among other results, and without using any gradient structure (e.g. \cite{AMP,HKST}) we extend Varopoulos' theorem to any $p \ge 1$ and prove the following:  

\begin{theorem}\label{sobo theorem intro}
Let $(X,\mu,\mathcal{E},\mathcal{F})$ be a Dirichlet space. Let $\{P_{t}\}_{t\in[0,\infty)}$ denote the Markovian semigroup associated with $(X,\mu,\mathcal{E},\mathcal{F})$.  Let $p \ge 1$.
\begin{enumerate}
\item Assume that $P_t$ admits a measurable heat kernel $p_t(x,y)$ satisfying, for some
$C>0$ and $\beta >0$,
\begin{equation}\label{eq:subGauss-upper3 intro}
p_{t}(x,y)\leq C t^{-\beta}
\end{equation}
for $\mu\times\mu$-a.e.\ $(x,y)\in X\times X$, and for each $t\in\bigl(0,+\infty \bigr)$. 
\item Assume that there exist $\alpha >0$ and $C>0$ such that for every $f \in L^p(X,\mu)$
 \begin{align}\label{local to global}
\| f \|_{p,\alpha} \le C \liminf_{t \to 0} t^{-\alpha} \left( \int_X \int_X |f(x)-f(y) |^p p_t (x,y) d\mu(x) d\mu(y) \right)^{1/p},
\end{align}
where 
 \begin{align}\label{norm intro}
\| f \|_{p,\alpha}:=\sup_{t > 0} t^{-\alpha} \left( \int_X \int_X |f(x)-f(y) |^p p_t (x,y) d\mu(x) d\mu(y) \right)^{1/p}.
\end{align}
\end{enumerate}
Then, if  $0<\alpha <\beta $ and $ p < \frac{\beta}{\alpha} $, there exists a 
constant $C_{p,\alpha,\beta}>0$ such that for every $f \in L^p(X,\mu)$,
\[
\| f \|_{L^q(X,\mu)} \le C_{p,\alpha,\beta} \| f \|_{p,\alpha},
\]
where $q=\frac{p\beta}{ \beta-p \alpha}$.
\end{theorem}

Note that it is possible to verify condition~\eqref{eq:subGauss-upper3 intro} in many classical and fractal examples because, according to \cite{BCLS,CKS87,KigB},  it is equivalent to the Nash inequality 
\begin{equation}\label{Nash}
\| f \|_{L^2(X,\mu)}^{2+2/\beta} 
\leqslant 
C \mathcal{E}(f,f) \| f \|_{L^1(X,\mu)}^{2/\beta}.
\end{equation}

In the case $p=2$, we will prove that the condition \eqref{local to global} is always satisfied with $\alpha =\frac{1}{2}$ and that we have $\mathcal{E}(f,f)   \simeq \| f \|^2_{2,1/2}$. As a consequence,  Theorem \ref{sobo theorem intro} is indeed an extension of the  Varopoulos theorem. On the other hand, when applied to the Euclidean space equipped with the standard Dirichlet form $\mathcal{E}(f,f)=\| \, |  \nabla f  | \, \|^2_{L^2 (\mathbb{R}^n)}$, for every $p \ge 1$  the condition \eqref{local to global} is always satisfied with $\alpha =\frac{1}{2}$ and Theorem \ref{sobo theorem intro} yields all of the Sobolev inequalities \eqref{sobo-intro} because in that case we have  $\| f \|_{p,1/2} \simeq \| \, |  \nabla f  | \, \|_{L^p (\mathbb{R}^n)}$.

Theorem \ref{sobo theorem intro} is proved by applying the methods of \cite{BCLS} in the general framework of Dirichlet spaces. In a first step, we prove in Theorem \ref{pol} that the assumption \eqref{eq:subGauss-upper3 intro} alone implies a weak Sobolev inequality
 \[
\sup_{s \ge 0}\, s\, \mu \left( \{ x \in X\, :\, | f(x) | \ge s \} \right)^{\frac{1}{q}} \le C_{p,\alpha,\beta} \| f \|_{p,\alpha},
\]
where $q=\frac{p\beta}{ \beta -p \alpha}$. In a second step we prove that if the condition \eqref{local to global} is also assumed, then this weak Sobolev inequality implies the strong one:
\[
\| f \|_{L^q(X,\mu)} \le C_{p,\alpha,\beta} \| f \|_{p,\alpha}.
\]
The case $p=1$ is of special interest. In that case, weak Sobolev inequalities applied to indicator functions of sets are related to  isoperimetric type inequalities, and by adapting a method of M. Ledoux, we obtain in Proposition \ref{prop:iso} an isoperimetric type inequality with explicit constant.  
 We note that it is well-known since the works of Maz'ja and Federer-Fleming that the case $p=1$ in \eqref{sobo-intro} is equivalent  to the isoperimetric inequality in $\mathbb{R}^n$ (see \cite{Maz85}) and is intimately connected to the theory of BV functions. It is therefore expected that for $p=1$ and under assumption  \eqref{local to global} Theorem \ref{sobo theorem intro} should yield isoperimetric type inequalities and possibly also a theory of BV functions valid in very general Dirichlet spaces. Moreover, the seminorm $\| f \|_{1,\alpha}$ provides a natural notion of variation of a function in that context. In more restrictive contexts, this remark  is  made very precise in the papers \cite{ABCRST2} and \cite{ABCRST3}, where it is moreover shown that the condition \eqref{local to global} is satisfied if the underlying space $X$ satisfies a weak Bakry-\'Emery type curvature condition.

In view of Theorem \ref{sobo theorem intro}, given a Dirichlet space $(X,\mu,\mathcal{E},\mathcal{F})$ with semigroup  $\{P_{t}\}_{t\in[0,\infty)}$ it is natural to thoroughly study the family of Besov type spaces
\[
\mathbf{B}^{p,\alpha}(X)=\left\{ f \in L^p(X,\mu), \| f \|_{p,\alpha}:= \sup_{t >0} t^{-\alpha} \left( \int_X P_t (|f-f(y)|^p)(y) d\mu(y) \right)^{1/p}<+\infty \right\}.
\]

\subsection*{Structure of the paper}

This paper is structured as follows. In Section~2 we introduce the notations and recall some basic facts about Dirichlet forms and their associated heat semigroups. In Section~3 we describe the basic setup of Dirichlet forms and our heat semigroup-based
Besov spaces $\mathbf{B}^{p,\alpha}(X)$, and conclude with a metric characterization of these spaces under the hypothesis that the heat semigroup has a kernel with sub-Gaussian estimates; this characterization, due to Pietruska-Pa{\l}uba~\cite{P-P10}, does not play a major role in the arguments introduced in this paper, but is included here as an illustration and  because of its  usefulness in concrete examples.

Section~4 is devoted to obtaining fundamental properties of the Besov classes, including the Banach space property, reflexivity, interpolation properties 
and locality in time. We show that  certain of the Besov spaces are non-trivial, in particular by showing in~Proposition~\ref{prop:energyasbesov} that the Besov space $\mathbf{B}^{2,1/2}(X)$ is precisely the domain $\mathcal{F}$ of
the Dirichlet form, the analog of the classical Sobolev space $W^{1,2}(X)$. This is in contrast to the classical (metric-based) Besov space theory, where $B_{2,\infty}^1(\R^n)$ consists only of constant functions, see~\cite{Bre02}.
From the preceding one deduces by elementary convexity considerations that $\mathbf{B}^{p,1/2}(X)$ is dense in $L^p$ if $1\leq p\leq2$.   We also give examples that establish the range of possibilities for $\mathbf{B}^{p,1/2}(X)$ when $p>2$: in Proposition~\ref{Besov Rn} we describe a smooth setting in which $\mathbf{B}^{p,1/2}(X)$ contains $C^\infty_0(X)$, but in Corollary~\ref{singular Kusuoka} we provide a class of Dirichlet forms for which the Besov spaces $\mathbf{B}^{p,1/2}(X)$ are trivial (contain only constant functions) when $p>2$.  Moreover, we begin to analyze the relationship between the Besov spaces $\mathbf{B}^{p,\alpha}(X)$ and the domain of the fractional powers of the generator $L$ of the Dirichlet form, showing  in particular that $(-L)^s:\mathbf{B}^{p,\alpha}(X)\to L^p$ is bounded for $0<s<\alpha$.

In Section~5, we  prove that, when $1<p \le 2$, the heat semigroup is always continuous as an operator $ L^p(X,\mu) \to \B^{p,1/2}(X)$, see Theorem \ref{continuity Besov chapter 1}.  It is remarkable that this is true in any Dirichlet space without any further assumption. In~\cite{ABCRST2} and~\cite{ABCRST3}, we will see that for $p >2$, this continuity is valid under weak Bakry-\'Emery type curvature conditions.  We use this result to establish some refinements of our triviality and non-triviality results from Section~4 and summarize them in terms  of critical exponents for density and triviality of the Besov spaces.   The results of this section will play an important role in \cite{ABCRST3}. In particular, the Besov critical exponents on fractal sets will be shown to be closely related to the geometry of these sets.

In Section~6 we will consider Sobolev-type embedding theorems for the Besov classes $\mathbf{B}^{p,\alpha}(X)$. Our main assumption is that the underlying Dirichlet space admits a heat kernel $p_t(x,y)$ satisfying a global upper bound of the type $p_t(x,y) \le c t^{-\beta}$.  In Dirichlet spaces, the proof of the existence of a Sobolev inequality with $p=2$ under this type of heat kernel estimates goes back to a celebrated work by N.~Varopoulos (see Chapter~2 of~\cite{varopoulos} and the references therein). To study the case $p\neq 2$, we make use of the ideas and methods developed in~\cite{BCLS,saloff2002} and more recently in \cite{BK}. Those methods are general enough to apply to our setting and underline the fact that our Besov classes provide a natural framework for a general theory of BV functions and isoperimetric inequalities on arbitrary Dirichlet spaces. These outline the beginnings of a connection between our Besov classes and isoperimetric type estimates that will be further explored under various  assumptions in the works~\cite{ABCRST2,ABCRST3,ABCRST4}.
Among many others, one of the future applications of Besov classes and isoperimetry will 
be to study diffusions on pattern spaces of quasicrystals~\cite{A-RHTT}, 
 corresponding to a unique strongly local but not strictly local Dirichlet form for which energy measures are absolutely continuous. 


In Section~7 we give some further applications of the main ideas under the assumption of either a Poincar\'e inequality or a log-Sobolev inequality. The idea is to replace the ultracontractivity estimate assumption of Section~6 by a supercontractive or hypercontractive one and explore the corresponding isoperimetric information. Such results may potentially be applied in infinite-dimensional situations like the Wiener space.

We conclude this introduction by  noting some references from the existing literature that are closest to our work.   The literature on Besov spaces is so large that it is not possible to be exhaustive, but we hope the following may be helpful to the reader.  Further references and comments will be given throughout the text. For many equivalent descriptions of the  Besov-Nikol'skii  spaces in $\mathbb{R}^n$, including Poisson heat kernel characterizations we refer to~\cite{Taibleson}. For the classical theory of Besov spaces from the point of view of interpolation theory, we refer for  example to the book of Triebel~\cite{Trie}. The relationship between Besov spaces and the Laplace operator or its square root in different settings has been studied from various points of view for some time; see, for example, \cite{CSC91} on Lie groups,  \cite{BDY12} on spaces with polynomial upper bound on the volume and Poisson-type heat kernel bound, \cite{HZ} on fractal metric spaces, and \cite{GrigLiu15} on metric measure spaces with sub-Gaussian heat kernel estimates and certain regularity assumptions. We also note  that Besov spaces can be characterized via wavelet frames, see~\cite{CKP12,KP15}. Finally, our definition of Besov classes is particularly closely connected to the work of Pietruska-Pa{\l}uba~\cite{P-P10}, and we learned much about  the relevance of this approach to Dirichlet forms from work of Grigor'yan, Hu and Lau~\cite{GHL:TAMS2003,MR2743439}.

\subsection*{Acknowledgments}
The authors thank Naotaka Kajino for many stimulating and helpful discussions.
The authors also thank the anonymous referee for comments that helped improve the exposition of the paper.
P.A-R. was partly supported by the Feodor Lynen Fellowship, Alexander von Humboldt Foundation (Germany) the grant DMS \#1951577 and \#1855349 of the NSF (U.S.A.). 
F.B. was partly supported by the grant DMS \#1660031 of the NSF (U.S.A.) 
and a Simons Foundation Collaboration grant.  
L.R. was partly supported by the grant DMS \#1659643 of the NSF (U.S.A.).  
N.S. was partly supported by the grants DMS \#1800161 and \#1500440 of the NSF (U.S.A.). 
A.T. was partly supported by the grant DMS \#1613025 of the NSF (U.S.A.).


\section{Preliminaries}\label{sec2}

In this section we introduce the notations and notions used throughout the paper and, for convenience of the reader, collect some standard definitions and known results that will be used later. The book~\cite{FOT} is a classical reference on the theory of Dirichlet forms and we refer to it for further details. We also refer to \cite{BouleauHirsch} for an exposition of the theory that does not use the hypothesis of regularity of the form. For the general theory of heat semigroups we refer for instance to~\cite{EBD} or~\cite{EN}.

\subsection{Dirichlet forms}

Throughout the paper, let $X$ be a good measurable space (like a Polish space) equipped with a $\sigma$-finite measure $\mu$. By good measurable space, we mean a measurable space for which the measure decomposition theorem holds and for which there exists a countable family generating the $\sigma$-algebra of $X$ (see~\cite[page 7]{BGL} for a discussion about good measurable spaces).

 Let $(\mathcal{E},\mathcal{F}=\mathbf{dom}(\mathcal{E}))$ be a densely defined closed symmetric form on $L^2(X,\mu)$. A function $v$ on $X$ is called a normal contraction of the function $u$ if for almost every $x,y \in X$,
\[
| v(x)-v(y)| \le |u(x) -u(y)| \text{    and    } |v(x)| \le |u(x)|.
\]
The form $\mathcal{E}$ is called a Dirichlet form if it is Markovian, that is, it has the property that if $u \in \mathcal{F}$ and $v$ is a normal contraction of $u$ then $v \in \mathcal{F}$ and $\mathcal{E}(v,v) \le \mathcal{E} (u,u)$. In this paper we always assume that $\mathcal{E}$ is a Dirichlet form and refer to  $(X,\mu,\mathcal{E},\mathcal{F})$ as a Dirichlet space. Some basic properties of Dirichlet forms are collected in~\cite[Theorem 1.4.2]{FOT}. In particular, we note that $\mathcal{F} \cap L^\infty(X,\mu)$ is an algebra and $\mathcal F$ is a Hilbert space with the $\mathcal{E}_1$-norm 
\begin{equation}\label{e-E1}
\|f\|_{\mathcal{E}_1}:=\left( \| f \|_{L^2(X,\mu)}^2 + \mathcal{E}(f,f) \right)^{1/2}.
\end{equation} 

The Dirichlet space  $(X,\mu,\mathcal{E},\mathcal{F})$ is called regular if  $X$ is a locally compact topological space,   $\mu$ is a Radon measure whose support is $X$ and $(X,\mu,\mathcal{E},\mathcal{F})$ admits a core. If we denote by $C_c(X)$ the space of continuous functions with compact support in $X$, we recall that a core for $(X,\mu,\mathcal{E},\mathcal{F})$ is a subset $\mathcal{C}$ of $C_c(X) \cap \mathcal{F}$ which is dense in $C_c(X)$ in the supremum norm and dense in $\mathcal{F}$ in the $\mathcal{E}_1$-norm.

 If  $\mathcal{E}$ is regular, 
 then  for every $f\in \mathcal F\cap L^{\infty}(X,\mu)$, we can define the energy measure $\nu_{f}$ in the sense of~\cite{Beurling-Deny}  through the formula
\[
\int_X \phi d\nu_{f} = \mathcal{E}(f\phi,f)-\frac12 \mathcal{E}(\phi,f^2), \quad \phi\in \mathcal F \cap C_c(X),
\]
see \cite[Theorem 4.3.11]{ChenFukushima}.
Then $\nu_{f}$ can be extended to all $f\in \mathcal F$ by truncation, that is, for each positive integer we consider
$f_n:=\max\{-n,\min\{n,f\}\}$, and set $\nu_f$ to be the weak limit of the sequence of measures $\nu_{f_n}$.

In this paper, most of the time we will not need to assume that $(X,\mu,\mathcal{E},\mathcal{F})$ is regular, so if the regularity assumption is needed, it will be stated out explicitly.

\subsection{Heat semigroup}

Let $\{P_{t}\}_{t\in[0,\infty)}$ denote the self-adjoint  semigroup on $L^2(X,\mu)$ associated with the Dirichlet space $(X,\mu,\mathcal{E},\mathcal{F})$ and $L$ the infinitesimal generator of $\{P_{t}\}_{t\in[0,\infty)}$; see~\cite[Section 1.4]{FOT}). The semigroup $\{P_{t}\}_{t\in[0,\infty)}$ is referred to as the heat semigroup on $(X,\mu,\mathcal{E},\mathcal{F})$.

  The following spectral theory lemma can be found in~\cite[Proposition~1.2.3]{BouleauHirsch} or in~\cite[Section 4]{Gri}. It shows that one can recover the Dirichlet form $\mathcal{E}$ and its domain from the semigroup $\{P_{t}\}_{t\in[0,\infty)}$.

\begin{lemma}\label{lem:Dirich-from-Pt}
Denoting $\langle \cdot, \cdot \rangle$ as the inner product in $L^2(X,\mu)$, for $f\in L^2(X,\mu)$ we have that
\[
0<t\mapsto \frac{1}{t}\langle (I-P_t)f,f\rangle
\]
is a decreasing function. Moreover, the limit $\lim_{t\to 0^+}\frac{1}{t}\langle (I-P_t)f,f\rangle$ exists if and only if $f \in \mathcal{F}$, in which case,
\begin{equation}\label{eq:construct-Lap}
\mathcal{E}(f,f)=\lim_{t\to 0^+}\frac{1}{t}\langle (I-P_t)f,f\rangle.
\end{equation}
\end{lemma}

By definition, the semigroup $\{P_{t}\}_{t\in[0,\infty)}$ acts on $L^2(X,\mu)$. However, it inherits from the Markovian property of the Dirichlet form the sub-Markovian property: if $0 \le f \le 1$ then $0 \le P_t f \le 1$. This fundamental property allows us to develop an $L^p$ theory of the semigroup and from this classical theory (see for instance~\cite[Theorem~1.4.1 and~1.4.2]{EBD}), the following properties of the semigroup  $\{P_{t}\}_{t\in[0,\infty)}$ are known:
\begin{itemize}
\item The semigroup $\{P_{t}\}_{t\in[0,\infty)}$ maps $L^1(X,\mu) \cap L^\infty (X,\mu)$ into itself and may be extended, using the  Riesz-Thorin interpolation, to a 
contraction semigroup on $L^p (X,\mu)$ for each $1 \le p \le +\infty$. We will denote that extension also by $P_t$. We explicitly note that the contraction property reads:
\[
\| P_t f \|_{L^p(X,\mu)} \le \| f \|_{L^p(X,\mu)}, \quad f \in L^p (X,\mu),  1 \le p \le \infty .
\]
The semigroup $\{P_{t}\}_{t\in[0,\infty)}$ is said to be conservative if $P_t 1=1$. 

\textbf{In this paper, we always assume that $\{P_t\}_{t\in[0,\infty)}$ is conservative.} 
This assumption is not overly restrictive, as it holds also for the standard Dirichlet form on the Wiener space,
see \cite{Kusuoka}.
\item The semigroup $\{P_{t}\}_{t\in[0,\infty)}$ is symmetric, i.e. for $1 \le p , q \le \infty$ with $\frac{1}{p}+\frac{1}{q}=1$,  $f \in L^p(X,\mu)$, $g \in L^q(X,\mu)$, $t \ge 0$,
\begin{align}\label{symmetry P}
\int_X (P_t f)(x) g(x) d\mu(x)=\int_X f(x) (P_t g)(x) d\mu(x).
\end{align}
\item  
The semigroup $\{P_{t}\}_{t\in[0,\infty)}$ is strongly continuous  on $L^p (X,\mu)$ for $1 \le p <+\infty$, i.e., 
\begin{align}\label{Strong continuity}
\|P_t f-f\|_{L^p(X,\mu)}\to 0, \quad \text{ as } t\to 0.
\end{align}
\item The semigroup $\{P_{t}\}_{t\in[0,\infty)}$ is an analytic semigroup on $L^p (X,\mu)$  for $1<p<+\infty$.  In particular, from \cite[page 101]{EN},   there exists a constant $C>0$ independent of $t>0$ (but depending on $p$) such that for every $f \in L^p(X,\mu)$,
\begin{align}\label{analytic bound}
\|LP_tf\|_{L^p(X,\mu)}\le \frac{C}{t}\|f\|_{L^p(X,\mu)}.
\end{align}
\end{itemize}

Since we  assume conservativeness, the semigroup $\{P_{t}\}_{t\in[0,\infty)}$ yields a family of heat kernel measures. Namely, from \cite[Theorem 1.2.3]{BGL}, for every bounded or non-negative measurable function $f:X \to \mathbb{R}$,
\begin{align}\label{heat kernel measure}
P_tf (x)=\int_X f(y) p_t(x,dy), \quad t \ge 0, x \in X,
\end{align}
where, for each $t>0$,  $p_t(x,dy)$ is a probability kernel (that is, for every $x \in X$, $p_t(x,\cdot)$ is a probability measure on $X$ and for every measurable set $A$, $x \to p_t(x,A)$ is measurable). Note that from the symmetry property of the heat semigroup, the measure defined on $X \times X$ by $\nu_t(A \times B)=\int_X 1_A P_t 1_B d\mu$ is symmetric, thus one has for every 
non-negative measurable function $F: X \times X \to \mathbb{R}$, 
\begin{align}\label{heat kernel measure symmetry}
\int_X  \int_X  F(x,y) p_t(x,dy) d\mu(x)=\int_X  \int_X F(x,y)  p_t(y,dx) d\mu(y).
\end{align}

We say that the semigroup $\{P_{t}\}_{t\in[0,\infty)}$ admits a heat kernel if the heat kernel measures have a density with respect to $\mu$, i.e. there exists a measurable function $p:\mathbb{R}_{>0} \times X \times X \to \mathbb{R}_{\ge 0}$, 
(and we denote $p(t,x,y)$ as $p_t(x,y)$ for $t>0$ and $x,y\in X$)
such that for every $t >0, x,y \in X$, $f \in L^p(X,\mu)$, $ 1 \le p \le \infty$,
\begin{align*}
P_t f(x) =\int_X p_t (x,y) f(y) d\mu(y).
\end{align*}
Many of our results do not require the existence of a heat kernel. The major exceptions are the Sobolev embeddings in Section~\ref{section sobolev}. This assumption will thus be explicitly stated when needed.

%

The following lemma is well known. It follows from the classical Jensen's inequality applied to \eqref{heat kernel measure}.

\begin{lemma}\label{lem:Pt-Holder}
Let $\Phi:\R\to[0,\infty)$ be a convex function. For $1\le p<\infty$ and all $f\in L^p(X,\mu)$ and $t \ge 0$ we have 
\[
\Phi(P_t(f))\le P_t(\Phi\circ f).
\]
In particular, for $1\le p<\infty$ and all $f\in L^p(X,\mu)$  and $t \ge 0$ we have
\[
|P_t(f) |^p \le P_t(|f|^p).
\]
\end{lemma}

%
%
%

\section{Heat semigroup-based Besov spaces }\label{sec:definition}

Let $(X,\mu,\mathcal{E},\mathcal{F})$ be a Dirichlet space and let $\{P_{t}\}_{t\in[0,\infty)}$ denote the associated  heat semigroup. As already pointed out, $\{P_{t}\}_{t\in[0,\infty)}$ is assumed to be conservative. Our basic definition of the (heat semigroup-based) Besov seminorm is the following:

\begin{definition}\label{def:Besov}
Let $p \ge 1$ and $\alpha \ge 0$. For $f \in L^p(X,\mu)$, we define the Besov  seminorm:
\[
\| f \|_{p,\alpha}= \sup_{t >0} t^{-\alpha} \left( \int_X P_t (|f-f(y)|^p)(y) d\mu(y) \right)^{1/p}.
\]
\end{definition}

Observe that in terms of the heat kernel measure \eqref{heat kernel measure}, one has:
\[
 \int_X P_t (|f-f(y)|^p)(y) d\mu(y)= \int_X \int_X |f(x)-f(y) |^p p_t (y, dx)  d\mu(y).
\]
Our goal in this paper  is to study the Besov type  spaces
\begin{align}\label{eq:def:Besov}
\mathbf{B}^{p,\alpha}(X)=\{ f \in L^p(X,\mu)\, :\,  \| f \|_{p,\alpha} <+\infty \}.
\end{align}
The norm on $\mathbf{B}^{p,\alpha}(X)$ is defined as:
\[
\| f \|_{\mathbf{B}^{p,\alpha}(X)} =\| f \|_{L^p(X,\mu)} + \| f \|_{p,\alpha}.
\]

\begin{remark}\label{rem-normalcontruction}
It is apparent that if $v$ is a normal contraction of $u\in \mathbf{B}^{p,\alpha}(X)$ then $v\in \mathbf{B}^{p,\alpha}(X)$ with 
$\|v\|_{p,\alpha}\le \|u\|_{p,\alpha}$ and $\|v\|_{\mathbf{B}^{p,\alpha}(X)}\le \|u\|_{\mathbf{B}^{p,\alpha}(X)}$.
This fact will be used from time to time without further comment.
\end{remark}

One has first the following example of the standard Dirichlet form on $\mathbb{R}^n$.

\begin{example}
If $X=\mathbb{R}^n$ and $\mathcal{E}$ is the standard Dirichlet form on $\mathbb{R}^n$, that is, for $f,g \in W^{1,2}(\mathbb{R}^n)$
we have 
\[
\mathcal{E}(f,g) =\int_{\mathbb{R}^n} \langle \nabla f (x) ,\nabla g (x) \rangle dx, 
\]
then, for $p \ge 1$ and $\alpha \ge 0$ the class $\mathbf{B}^{p,\alpha}(X)$ coincides with the  Besov-Nikol'skii class $B^{2\alpha}_{p,\infty} (\mathbb{R}^n)$ that consists of 
$f \in L^p(\mathbb{R}^n,dx)$ such that
\[
\sup_{h \in \mathbb{R}^n, h \neq 0} \frac{ \| f( \cdot+h)-f(\cdot) \|_p}{|h|^{2\alpha}} <+\infty.
\]
We refer, for instance, to \cite{AS} and \cite[Theorems~4 and~4*]{Taibleson}  
for several equivalent descriptions of those spaces.
\end{example}

Comparable Besov type spaces have previously been considered in the literature in some specific settings. A most relevant reference  is the paper by K.~Pietruska-Pa{\l}uba \cite{P-P10} (see also references therein). In particular, \cite{P-P10} provides a metric characterization of the spaces $\mathbf{B}^{p,\alpha}(X)$ on Dirichlet spaces that admit a heat kernel with Gaussian or sub-Gaussian heat kernel estimates.

\begin{theorem}[{\cite[Theorem~3.2]{P-P10}}]\label{sub gaussian intro}
Let $(X,\mu,\mathcal{E},\mathcal{F})$ be a 
Dirichlet space and let $d$ be a metric on $X$ compatible with
the topology of $X$. Assume that the metric space $(X,d)$ is Ahlfors $d_H$-regular and that  $\{P_{t}\}_{t\in[0,\infty)}$
admits a heat kernel $p_{t}(x,y)$ satisfying, for some
$c_{1},c_{2}, c_3, c_4 \in(0,\infty)$ and $d_{W}\in(1,\infty)$,
\begin{equation*}
c_{1}t^{-d_{H}/d_{W}}\exp\biggl(-c_{2}\Bigl(\frac{d(x,y)^{d_{W}}}{t}\Bigr)^{\frac{1}{d_{W}-1}}\biggr) \le p_{t}(x,y)\leq c_{3}t^{-d_{H}/d_{W}}\exp\biggl(-c_{4}\Bigl(\frac{d(x,y)^{d_{W}}}{t}\Bigr)^{\frac{1}{d_{W}-1}}\biggr)
\end{equation*}
for $\mu\times\mu$-a.e.\ $(x,y)\in X\times X$ for each $t\in\bigl(0,+\infty\bigr)$. Let $p \ge 1$ and $\alpha \ge 0$. We have 
\[
\mathbf{B}^{p,\alpha}(X)=\left\{f\in L^{p}(X,\mu)\, :\,  \sup_{r>0} \frac{1}{r^{\alpha d_W+d_{H}/p}}
 \biggl(\iint_{\Delta_r} 
 |f(x)-f(y)|^{p}\,d\mu(x)\,d\mu(y)\biggr)^{1/p}<\infty\right\}
\]
with comparable seminorms, where for $r>0$ the set $\Delta_r$ denotes the collection of all $(x,y)\in X\times X$ for which
$d(x,y)<r$.
\end{theorem}

The further study of the spaces $\mathbf{B}^{p,\alpha}(X)$ on Dirichlet spaces that admit a heat kernel with sub-Gaussian  estimates will be the object of the paper \cite{ABCRST3}. In the present paper, one of the main goals is to study the spaces $\mathbf{B}^{p,\alpha}(X)$ in the framework of a general Dirichlet space $(X,\mu,\mathcal{E},\mathcal{F})$.

\section{Properties of the heat semigroup-based Besov spaces}\label{sec:properties}

In this section we identify and prove some fundamental properties of the Besov spaces given in Section \ref{sec:definition}, including Banach space property, reflexivity, and non-triviality. We also show that the supremum in the definition of Besov spaces can be replaced with limit supremum; this ``locality in time" property is very useful in obtaining local information about Besov functions (in particular, dimensions of boundaries of sets whose characteristic functions are in a Besov class from their norms, see \cite{ABCRST2,ABCRST3,ABCRST4}). We will also prove interpolation inequalities and pseudo-Poincar\'e inequalities. Those pseudo-Poincar\'e inequalities will play a prominent role in the study of Sobolev inequalities, see Section \ref{section sobolev}. Finally, we study the relationship between the Besov spaces and the domain of the fractional powers of the generator of the Dirichlet form.

Throughout the section, let $(X,\mu,\mathcal{E},\mathcal{F})$ be a Dirichlet space and let $\{P_{t}\}_{t\in[0,\infty)}$ denote the associated  Markovian semigroup.

\subsection{Locality in time}\label{locality in time section}

The following ``locality in time" property will be useful in understanding functions of bounded variation, to be studied in the second and third paper \cite{ABCRST2,ABCRST3}. It also underlines the fact that the Besov 
energy seminorm $\|\cdot\|_{p,\alpha}$ is a global object, since in going from the supremum in the Besov norm to 
limit supremum one also picks up the $L^p$-norm. Recall the definition of (heat semigroup-based) Besov classes from
Definition~\ref{def:Besov}.

\begin{lemma}\label{Lemma limsup debut}
Let $p\ge1$ and  $\alpha \ge 0$. Then
\[
\mathbf{B}^{p,\alpha}(X)
=\left\{ f \in L^p(X,\mu)\, :\,  \limsup_{t \to 0} t^{-\alpha} \left( \int_X P_t (|f-f(y)|^p)(y) d\mu(y) \right)^{1/p} <+\infty \right\}.
\]
Moreover, if $\beta > \alpha$, then $\mathbf{B}^{p,\beta}(X) \subset \mathbf{B}^{p,\alpha}(X)$.
Furthermore, for $f \in \mathbf{B}^{p,\alpha}(X)$, and for every $t >0$, we have
\[
\| f \|_{p,\alpha} \le \frac{2}{t^\alpha} \| f \|_{L^p(X,\mu)} +\sup_{s\in (0,t]} s^{-\alpha} \left( \int_X P_s (|f-f(y)|^p)(y) d\mu(y) \right)^{1/p}.
\]
\end{lemma}

\begin{proof}
The claim $\mathbf{B}^{p,\beta}(X) \subset \mathbf{B}^{p,\alpha}(X)$ for $\beta>\alpha$ is immediate.

If $f \in \mathbf{B}^{p,\alpha}(X)$, then
\[
\limsup_{t \to 0} t^{-\alpha} \left( \int_X P_t (|f-f(y)|^p)(y) d\mu(y) \right)^{1/p} \le \| f \|_{p,\alpha}.
\]
Conversely, if $\limsup_{t \to 0} t^{-\alpha} \left( \int_X P_t (|f-f(y)|^p)(y) d\mu(y) \right)^{1/p} <+\infty$, then
there is some $\ve >0$ for which
\[
\sup_{t \in (0,\ve]} t^{-\alpha} \left( \int_X P_t (|f-f(y)|^p)(y) d\mu(y) \right)^{1/p} <\infty.
\]
For $t >\ve$, since $|f(x)-f(y)|^p \le 2^{p-1} (|f(x)|^p+|f(y)|^p)$, the semigroup is conservative (and hence
$P_t1(x)=1$ for all $x\in X$) and $\int_XP_t(|f|^p)(x)\, d\mu(x)\le \int_X|f(x)|^p\, d\mu(x)$,
we have
\begin{align*}
 t^{-\alpha} \left( \int_X P_t (|f-f(y)|^p)(y) d\mu(y) \right)^{1/p} \le 2 \ve^{-\alpha} \| f \|_{L^p(X,\mu)}.
\end{align*}
The last inequality stated in the lemma now follows from the above inequality, and we also have that
if $\limsup_{t \to 0} t^{-\alpha} \left( \int_X P_t (|f-f(y)|^p)(y) d\mu(y) \right)^{1/p} <\infty$, then
$f\in \mathbf{B}^{p,\alpha}(X)$. This completes the proof. 
\end{proof}

Interestingly, in a large class of examples of strictly local Dirichlet forms, for $\alpha=1/2$,
\[
\limsup_{t \to 0} t^{-\alpha} \left( \int_X P_t (|f-f(y)|^p)(y) d\mu(y) \right)^{1/p}
\]
is actually a limit that can be explicitly computed:

\begin{proposition}\label{Besov Rn}
Assume that $X$ is a smooth manifold  of dimension $n\ge d$. Let $L=V_0+\sum_{i=1}^d V_i^2$ be a H\"ormander's type operator on $X$, where the $V_i$'s are smooth vector fields. Let us assume that $L$ is essentially self-adjoint on $C_0^\infty(X)$ in $L^2(X,\mu)$ for some Radon measure $\mu$ on $X$. Consider the Dirichlet space $(X,\mu, \mathcal{E},\mathcal{F})$ obtained by closing the pre-Dirichlet form
\[
\mathcal{E}(f,g)=\int_X \Gamma(f,g) d\mu(x), \quad f,g \in C_0^\infty(X),
\]
where $\Gamma(f,g)$ is the carr\'e du champ operator defined by  $\Gamma(f,g)=\frac{1}{2}( L(fg)-fLg -gLf) =\sum_{i=1}^d (V_i f)( V_i g )$.  Assume that the associated semigroup $P_t$ is conservative.
Then, for every $p \ge 1$, $$C_0^\infty (X) \subset \mathbf{B}^{p,1/2}(X)$$ and one has for every $f \in C_0^\infty (X) $, and open set $A \subset X$,
\[
\lim_{t \to 0} t^{-1/2} \left( \int_A P_t (|f-f(x)|^p)(x) d\mu(x) \right)^{1/p} = 2\left( \frac{ \Gamma \left( \frac{1+p}{2} \right) }{\sqrt
{\pi}}\right)^{1/p} \left( \int_A \Gamma ( f,f) (x)^{p/2} d\mu(x)\right)^{1/p},
\]
where $\Gamma \left( \frac{1+p}{2} \right) $ denotes the Euler's gamma function.
\end{proposition}

We will prove this proposition below, after discussing its consequences.

\begin{remark}
In the previous setting, one therefore has
\begin{equation}\label{inequality pre}
\left( \int_X \Gamma ( f,f) (x)^{p/2} d\mu(x)\right)^{1/p} \le C_p \| f \|_{p,1/2}
\end{equation}
 for $f \in C_0^\infty(X)$. Moreover, if $P_t$ satisfies the Bakry-\'Emery estimate $\sqrt{\Gamma(P_t f)} \le C P_t \sqrt{\Gamma( f)} $, we will 
see in~\cite[Section 4.5]{ABCRST2}
that the converse inequality to~\eqref{inequality pre} holds for $p=1$, which takes the form
\[
\| f \|_{1,1/2} \le c \left( \int_X \Gamma ( f,f) (x)^{1/2} d\mu(x)\right).
\]
\end{remark}

\begin{remark}
Proposition \ref{Besov Rn} indicates that at a high level of generality, one may  expect the  Besov spaces $ \mathbf{B}^{p, 1/2}(X)$, $1 \le p <+\infty$  to be closely related to the various notions of Sobolev spaces that have been defined on metric measure spaces 
{\rm(}see for instance \cite{Shan2000}{\rm)}. While in this paper we shall only be concerned with the study of all the Besov  spaces $\mathbf{B}^{p,\alpha}(X)$, the comparison between Sobolev spaces and Besov spaces will be made in  \cite[Section 7]{ABCRST2}  in the framework of Dirichlet spaces with absolutely continuous energy measures. In the framework of \cite{ABCRST3},  it will be interesting to compare our results with the recent preprint \cite{HinzKochMeinert} on Sobolev spaces and calculus of variations on fractals. Such a comparison will be the subject of future study.
\end{remark}

\begin{example}
If $X=\mathbb{R}^n$ and $\mathcal{E}$ is the standard Dirichlet form on $\mathbb{R}^n$, it is natural to expect that for every $p \ge 1$, and every $f \in \mathbf{B}^{p,1/2}(X)$
\[
 \lim_{t \to 0} t^{-1/2} \left( \int_{\mathbb{R}^n} P_t (|f-f(x)|^p)(x) dx \right)^{1/p}=2\left( \frac{ \Gamma \left( \frac{1+p}{2} \right) }{\sqrt
{\pi}}\right)^{1/p} \left( \int_{\mathbb{R}^n} | \nabla f (x)|^p dx\right)^{1/p}.
\]
The case $p=1$ is proved in \cite{MPPP}, but we did not find it in the literature for $p>1$, $p \neq 2$, though it seems to be closely related to  \cite{BBM} . 
\end{example}

\begin{proof} [\textbf{Proof of Proposition \ref{Besov Rn}}]
We use here a probabilistic argument. For $x \in X$, we denote by $(B_t^x)_{t \ge 0}$ the $L$-Brownian motion on $X$ started from $x$, that is the diffusion with generator $L$. It can be constructed as the solution of a stochastic differential equation in Stratonovich form:
\[
dB_t^x=V_0(B_t^x)dt +\sqrt{2}\sum_{i=1}^d V_i(B_t^x)\circ d\beta^i_t
\]
where $\beta$ is a $d$-dimensional Brownian motion. 
Let $f \in C_0^\infty(X)$. The process
\[
M_t^f =f(B_t^x) -f(x) -\int_0^t L f(B_s^x) ds
\]
is a square integrable martingale that can be written
\[
M_t^f=\sqrt{2} \sum_{i=1}^d \int_0^t (V_i f)(B_s^x)d\beta^i_s.
\]
We have then
\begin{align*}
P_t (|f-f(x)|^p)(x) & =\mathbb{E} \left( | f(B_t^x) -f(x) |^p\right) \\
 &=\mathbb{E} \left( \left| M_t^f +\int_0^t L f(B_s^x) ds \right|^p \right).
\end{align*}
Observe now that $\frac{1}{\sqrt{t}} \int_0^t L f(B_s^x) ds$ almost surely converges to 0 when $t \to 0$. Note also that 
$\frac{1}{\sqrt{t}} M_t^f$ converges in all $L^p$'s to the Gaussian random variable 
$\sqrt{2} \sum_{i=1}^d (V_i f) (x) \beta^i_1$. Since $f$ has a compact support, one deduces that
\[
\lim_{t \to 0} t^{-1/2} \left( \int_A P_t (|f-f(x)|^p)(x) d\mu(x) \right)^{1/p} = C_p \left( \int_A \Gamma ( f,f) (x)^{p/2} d\mu(x)\right)^{1/p},
\]
with $C_p=\sqrt{2} \mathbb{E}(| N|^p)^{1/p}=2\left( \frac{ \Gamma \left( \frac{1+p}{2} \right) }{\sqrt
{\pi}}\right)^{1/p} $, where $N$ denotes a Gaussian random variable with mean 0 and variance 1.
\end{proof}

%

\subsection{$\mathbf{B}^{2,1/2}(X)=\mathcal{F}$ and non-triviality of some of the spaces $\mathbf{B}^{p,\alpha}(X)$}

We prove that the Besov space $\mathbf{B}^{2,1/2}(X)$ is exactly the domain $\mathcal{F}$ of the Dirichlet form.  It follows that $\mathbf{B}^{2,1/2}(X)$ is dense in $L^2(X,\mu)$. 

\begin{proposition}\label{prop:energyasbesov}
We have \/ $\mathbf{B}^{2,1/2}(X)=\mathcal{F}$. Moreover, for every $f \in \mathcal{F}$,  $2\DF(f,f)=\|f\|_{2,1/2}^2$.
\end{proposition}

\begin{proof}
Let $f \in L^2(X,\mu)$. Note that as $P_t$ is linear and fixes constant functions (by its conservativeness), we have for $t>0$  that
\[
P_t(|f-f(y)|^2)(y)=P_t(f^2)(y)+f(y)^2-2f(y)P_t(f)(y).
\]
Therefore,
\begin{align*}
 \frac1{2t} \int_X P_t(|f-f(y)|^2)(y)\, d\mu(y)&=\frac1{2t} \int_X \left( P_t(f^2)(y)+f(y)^2-2f(y)P_t(f)(y)\right)\, d\mu(y).
 \end{align*}
 Now using the symmetry \eqref{symmetry P} and the conservativeness of $\{P_t\}_{t\in [0,\infty)}$, we have
 \[
 \int_X  P_t(f^2)(y)\, d\mu(y)=\int_X  (P_t 1)(y)f^2(y)\, d\mu(y)=\int_X f^2(y)\, d\mu(y).
 \]
 Therefore,
\begin{align}\label{prop:energyasbesoveq}
 \frac1{2t} \int_X P_t(|f-f(y)|^2)(y)\, d\mu(y)&=\frac1{t} \int_X \left( f(y)^2-f(y)P_t(f)(y)\right)\, d\mu(y) \notag \\
  & = \frac1t \langle (I-P_t)f,f\rangle.
 \end{align}
From Lemma~\ref{lem:Dirich-from-Pt} above, one sees
that the right side of~\eqref{prop:energyasbesoveq} is positive and decreasing as $t$ increases, and has limit 
$\DF(f,f)$ as $t\downarrow0$ if and only if $f\in\mathcal{F}$. From this, we know that the limit $t\to0^+$ of the left-hand side term above
is the supremum, and the claim follows.
\end{proof} 

\begin{proposition}\label{prop:convexityembed}
If $1\leq q\leq p < \infty$ and  $f\in \mathbf{B}^{p,\alpha}(X)$, then $|f|^{p/q}\in \mathbf{B}^{q,\alpha}(X)$ and
\begin{equation}\label{eqn:convexitybound}
\| |f|^{p/q}\|_{q,\alpha} \leq 2^{1/q}\brak{\frac{p}{q}}  \|f\|_{L^p(X,\mu)}^{(p/q)-1} \|f \|_{p,\alpha}.
\end{equation}
\end{proposition}
\begin{proof} 
We need only prove the seminorm estimate, as the norm estimate then follows trivially from H\"older's inequality. 
We use  for any $a,b>0$ such that $a\neq b$, the elementary inequality
\[
\frac{|a^{p/q}-b^{p/q}|}{|a-b|}\leq \frac{p}{q} \max\{a,b\}^{\frac{p}{q}-1}.
\]
Equivalently, 
\[
|a^{p/q}-b^{p/q}|^q \leq \brak{\frac{p}{q}}^q \max\{a,b\}^{p-q} |a-b|^{q}.
\]
Using this elementary inequality, one has
 \begin{align}\label{espn:ine3}
	 & \lefteqn{\int_X \int_X \bigl| |f(x)|^{p/q}-|f(y)|^{p/q}\bigr|^q \,p_t(y,dx)\,d\mu(y) } \quad& \notag \\
	&\leq \brak{\frac{p}{q}}^q \int_X \int_X  \bigl( |f(x)|^{p-q}+|f(y)|^{p-q}\bigr)\bigl| |f(x)|-|f(y)|\bigr|^q \,p_t(y,dx)\,d\mu(y)  \notag \\
	&\leq \brak{\frac{p}{q}}^q \int_X \int_X  \bigl( |f(x)|^{p-q}+|f(y)|^{p-q}\bigr)\bigl| f(x)-f(y)\bigr|^q \,p_t(y,dx)\,d\mu(y).
\end{align}
We now observe that thanks to 	 \eqref{heat kernel measure symmetry},
\begin{align*}
  \int_X \int_X |f(x)|^{p-q} \bigl| f(x)-f(y)\bigr|^q\, p_t(y,dx) \,d\mu(y)=\int_X \int_X   |f(x)|^{p-q}   \bigl| f(x)-f(y)\bigr|^q\, p_t(x,dy) \,d\mu(x). 
\end{align*}
On the other hand,
\begin{align*}
  \int_X \int_X |f(y)|^{p-q} \bigl| f(x)-f(y)\bigr|^q \,p_t(y,dx) \,d\mu(y)=\int_X \int_X   |f(x)|^{p-q}   \bigl| f(x)-f(y)\bigr|^q \,p_t(x,dy) \,d\mu(x). 
\end{align*}
 Thus, applying H\"older's inequality and then \eqref{heat kernel measure}  to~\eqref{espn:ine3}, we have
 \begin{align*}	
  & \lefteqn{\int_X \int_X \bigl| |f(x)|^{p/q}-|f(y)|^{p/q}\bigr|^q \,p_t(y,dx)\,d\mu(y) } \quad&\\
	&\leq 2\brak{\frac{p}{q}}^q \int_X  |f(x)|^{p-q} \left( \int_X \bigl| f(x)-f(y)\bigr|^q \,p_t(x,dy)\right) \,d\mu(x) \\
	&\leq 2\brak{\frac{p}{q}}^q \int_X |f(x)|^{p-q} \Bigl(\int_X  \bigl| f(x)-f(y)\bigr|^p \,p_t(x,dy)\Bigl)^{q/p} \, d\mu(x)\\
	&\leq 2\brak{\frac{p}{q}}^q \|f\|_{L^p(X,\mu)}^{p-q} \Bigl( \int_X \int_X  \bigl| f(x)-f(y)\bigr|^p \,p_t(x,dy)\, d\mu(x) \Bigl)^{q/p}\\
	&\leq 2 \brak{\frac{p}{q}}^q \|f\|_{L^p(X,\mu)}^{p-q} \left( \int_X P_t (|f-f(x)|^p)(x) d\mu(x)\right)^{q/p} \\
	&\leq 2\brak{\frac{p}{q}}^q \|f\|_{L^p(X,\mu)}^{p-q} t^{\alpha q} \| f\|_{p,\alpha}^q,
	\end{align*}
%
which implies~\eqref{eqn:convexitybound}.
\end{proof}

The following is an immediate consequence of Propositions~\ref{prop:energyasbesov} and~\ref{prop:convexityembed}.
\begin{corollary}\label{cor:densitybelow2}
If $1\leq q\leq p < \infty$ and $\mathbf{B}^{p,\alpha}(X)$ is dense in $L^p(X,\mu)$, then $ \mathbf{B}^{q,\alpha}(X)$ is dense in $L^q(X,\mu)$.  Hence $\mathbf{B}^{p,1/2}(X)$ is dense in $L^p(X,\mu)$ for $1\leq p\leq2$.
\end{corollary}

We note that when the measure is finite a stronger statement is true:

\begin{proposition}
Let us assume that $\mu(X)<\infty$. Then $p\geq q$ implies $\mathbf{B}^{p,\alpha}(X)\subset \mathbf{B}^{q,\alpha}(X)$ and
\[
\| f \|_{q,\alpha} \le \mu(X)^{1/q-1/p} \| f \|_{p,\alpha}.
\]
\end{proposition}

\begin{proof}
Let $ f \in \mathbf{B}^{p,\alpha}(X)$. From Lemma \ref{lem:Pt-Holder} and H\"older's inequality, one has 
\begin{align*}
\int_X P_t(|f-f(y)|^q) d\mu(y)  \leq \int_X P_t(|f-f(y)|^{p})^{q/p} d\mu(y) \leq (\mu(X))^{1-q/p} \left( \int_X P_t(|f-f(y)|^{p}) d\mu(y)\right)^{q/p}.
\end{align*}
\end{proof}


\subsection{Triviality of some of the spaces $\mathbf{B}^{p,\alpha}(X)$}

As we have seen, the space $\mathbf{B}^{2,1/2}(X)$ is dense in $L^2(X,\mu)$ since it is the domain 
$\mathcal{F}$ of $\mathcal{E}$ which is dense in $L^2(X)$. For other values of the parameters, it turns out that some of the spaces 
$\mathbf{B}^{p,\alpha}(X)$ are in general trivial.

\begin{proposition}\label{prop:BesovCEp}
Suppose that for all $f\in\mathcal{F}$ we have that $f$ is constant whenever $\mathcal{E}(f,f)=0$. 
Then, any $f \in \mathbf{B}^{p,\alpha }(X)$ with $1 \le p \le 2$ and $\alpha >1/p$ is constant.
\end{proposition}

\begin{proof}
Let $f \in \mathbf{B}^{p,\alpha }(X)$ with $1 \le p \le 2$. For $n \ge 0$, we set 
$f_n:=\min\{n,\max\{-n,f\}\}$.
Since $|f_n(x) -f_n(y)| \le |f(x) -f(y)|$ for every $x,y \in X$ and therefore $P_t(|f_n-f_n(y)|^p)\le P_t(|f-f(y)|^p)$,
 it is clear that $f_n \in \mathbf{B}^{p,\alpha }(X)$. Moreover, 
\[
P_t(|f_n-f_n(x)|^2) =P_t(|f_n-f_n(x)|^{2-p} |f_n-f_n(x)|^p)  \le 2^{2-p} \| f_n \|^{2-p}_{L^\infty(X,\mu)} P_t(|f_n-f_n(x)|^p).
\]
Therefore,
\[
 \frac1{2t} \int_X P_t(|f_n-f_n(x)|^2)(x) d\mu(x) \le  2^{1-p} t^{\alpha p -1}\| f_n \|^{2-p}_{L^\infty(X,\mu)} \| f_n\|_{p,\alpha}^p.
 \]
 As $\alpha p>1$, this implies that
 \[
 \lim_{t \to 0} \frac1{2t} \int_X P_t(|f_n-f_n(x)|^2)(x) d\mu(x)  =0.
 \]
 Thus $f_n\in \mathbf{B}^{2,1/2}(X)$. Hence from Lemma~\ref{lem:Dirich-from-Pt} and 
 Proposition~\ref{prop:energyasbesov},
we see that
 $f_n \in \mathcal{F}$ and $\mathcal{E}(f_n,f_n)=0$. This implies that $f_n$ is constant for every $n$, thus $f$ is constant.
 \end{proof}

The following theorem says that functions in $\mathbf{B}^{p,1/2}(X)\cap\mathcal{F}$  have a property related to the  carr\'e du champ.
For the definition of a regular Dirichlet space and energy measure used in the proof we refer to the preliminaries 
in Section~\ref{sec2}. 

\begin{theorem}\label{thm-ac}
Let $p>2$. If $f\in\mathbf{B}^{p,1/2}(X)\cap\mathcal{F}$ then there is $\Gamma(f)\in L^1(X,\mu)$ such that for all $g\in L^\infty(X,\mu)\cap\mathcal{F}$, 
\begin{equation}\label{eqn-ac}
\int_X g \Gamma (f) d\mu =2 \mathcal{E}(gf,f)-\mathcal{E}(f^2,g).
\end{equation}
\end{theorem}
\begin{proof}
According to  \cite[Theorem A.4.1(ii)]{FOT}, any  Dirichlet space  $(X,\mu,\mathcal{E},\mathcal{F})$
satisfying our assumptions  is equivalent to a regular  Dirichlet space  $(X',\mu',\mathcal{E}',\mathcal{F}')$ where $\mu'$ is a Radon measure.
This result was first obtained in~\cite{Fukushima71}, and in~\cite[Section~6]{HKT}  the isomorphism is realized as a Gelfand transform.  According to~\cite[Appendix A.4
]{FOT}, the equivalence between the  Dirichlet spaces  $(X,\mu,\mathcal{E},\mathcal{F})$
and  
$(X',\mu',\mathcal{E}',\mathcal{F}')$
implies the equivalence of $L^p(X,\mu)$ and $L^p(X',\mu')$ spaces  and the equivalence of corresponding  semigroups $P_t$ and $P_t'$. The spaces $ \mathbf{B}^{p,1/2}(X)$ and $\mathbf{B}^{p,1/2}(X')$ are therefore also equivalent.

Since $(X',\mu',\mathcal{E}',\mathcal{F}')$ is regular,  the Radon energy measure $\nu'_{f'}$ exists for any $f'\in\mathcal{F}'$.
Let $f' \in \mathbf{B}^{p,1/2}(X')\cap\mathcal{F'}$ and suppose $g' \in L^\infty (X',\mu')\cap \mathcal{F'}$ 
with support of finite $\mu'$-measure. 
Note that then $g'\in L^{p/(p-2)}(X',\mu')$.

By H\"older's inequality from Lemma~\ref{lem:Pt-Holder}, we have
\begin{align*}
 \frac1t\int_{X'} |g'(y)| P'_t(|f'-f'(y)|^2)(y)\,d\mu'(y)
 &\leq \frac1t\int_{X'} |g'(y)| \bigl( P_t' (|f'-f'(y)|^p)(y)\bigr)^{2/p}\,d\mu'(y) \\
 &\leq \frac1t\biggl( \int_{X'} P_t'(|f'-f'(y)|^p)(y) \,d\mu'(y) \biggr)^{2/p} \|g'\|_{L^{\frac{p}{p-2}}(X',\mu')}\\
 &\leq \|f'\|_{p,1/2}^2\|g'\|_{L^{\frac{p}{p-2}}(X',\mu')}.
 \end{align*}
Observe that  we have
\begin{align*}
\int_{X'} g'(y)P_t'(|f'-f'(y)|^2)(y)\,d\mu'(y)&=\langle P_t'((f')^2),g'\rangle-2\langle P_t'f', f'g'\rangle+\langle (f')^2,g'\rangle\\
 &=-\langle -P_t'((f')^2),g'\rangle-2\langle P_t'f', f'g'\rangle+2\langle f',f'g'\rangle-\langle (f')^2,g'\rangle\\
 &=-\langle (I'-P_t')(f')^2,g'\rangle+2\langle (I'-P'_t)f',f'g'\rangle.
\end{align*}
Now using the above identity and then taking the limit $t\to 0$, we obtain
\begin{align*}
	 \|f'\|_{p,1/2}^2\|g'\|_{L^{p/(p-2)}(X',\mu)}
 \geq \lim_{t\to 0}\left( \frac2t\langle (I'-P_t')f', f'g' \rangle - \frac1t\langle (I'-P_t') (f')^2, g\rangle\right)
 \\= 2\DF'(f'g',f')-\DF'((f')^2,g')= \int_{X'}  2g'\,d\nu'_{f'},
\end{align*}  
where, as in the previous result, the limit is by 
Lemma~\ref{lem:Dirich-from-Pt} above. 
The final equality is from the definition of $\nu'_{f'}$ and~\cite[Theorem 4.3.11]{ChenFukushima},
see also~\cite{Beurling-Deny}. 

In particular, if $E_1\subset E_2$ are of finite $\mu'$ measure and $\mathbf{1}_{E_1}\leq g'\leq \mathbf{1}_{E_2}$ then we obtain
\begin{equation*}
\nu'_{f'}(E_1) \leq \int_{X'} g'\,d\nu'_{f'} \leq \frac{1}{2} \| f'\|_{p,1/2}^2 \bigl( \mu'(E_2) \bigr)^{(p-2)/p}.
\end{equation*}

We wish to show that $\nu'_{f'}$-measure of a $\mu'$-null $X'$-Borel set is zero. Since both $\mu'$ 
and $\nu'_{f'}$ are $X'$-Radon measures,
it suffices to show this for a compact $\mu'$-null set $E_1$. For $U\supset E_1$ open with compact closure there is a continuous 
function $h$ satisfying $h=1$ on $E_1$ and $h=0$ on $X'\setminus U$. Then by the regularity of $\mathcal{E}'$, we can find 
$k\in\mathcal{F'}$ for which $\|h-k\|_\infty<1/3$ (see~\cite[Definition~1.3.10(iii)]{ChenFukushima}), at which point $g'=3((k\wedge 2/3)-1/3)\vee0$ satisfies the conditions of the above estimate with $E_2=\bar{U}$, the closure of $U$. 
Then
\begin{equation*}
	\nu'_{f'}(E_1)\leq \|f'\|^2_{p,1/2} \mu'(\bar{U})^{(p-2)/p}\leq \|f'\|^2_{p,1/2}\mu'(V)^{(p-2)/p}
	\end{equation*}
for any open $V$ containing $\bar{U}$. Thus $\nu'_{f'}(E_1)\leq \inf_V\|f'\|^2_{p,1/2}\mu'(V)^{(p-2)/p}$, with the infimum over 
all open sets $V$ containing $E_1$; this is zero by the outer regularity of $\mu'$ on $X'$.  Hence $\nu'_{f'}\ll\mu'$ with a  density $\frac{\nu'_{f'}}{\mu'}\in L^1(X',\mu')$.  However the equivalence of the Dirichlet forms and $L^p$ spaces then allows us to take $\Gamma(f)\in L^1(X,\mu)$ so that for $u\in\mathcal{F}\cap L^\infty(X,\mu)$, 
\begin{equation*}
	\int_X u\Gamma(f)\,d\mu= \int_{X'} u'\frac{\nu'_{f'}}{\mu'}\,d\mu'
	= 2\DF'(f'u',f')-\DF'((f')^2,g')= \int_{X'}  2u'\,d\nu'_{f'}
	= 2\DF(fu,f)-\DF(f^2,u).\qedhere
	\end{equation*}
\end{proof}

We deduce two corollaries.  The first uses the definition of a  carr\'e du champ operator, see~ \cite[Defintion 4.1.2 ]{BouleauHirsch}, which is that there is a map $f\mapsto\Gamma(f)$ on a $\DF_1$-dense subspace of $\mathcal{F}\cap L^\infty$ such that~\eqref{eqn-ac} holds. It shows that for $p>2$,  $\mathbf{B}^{p,1/2}(X)\cap\mathcal{F}$  is dense in $\mathcal{F}$ only in Dirichlet spaces that admit a carr\'e du champ operator.

\begin{corollary}
If $\mathbf{B}^{p,1/2}(X)\cap\mathcal{F}$ is dense in $\mathcal{F}$ with respect to the 
norm $\DF_1$ defined in~\eqref{e-E1} for some $p>2$, then $\DF$ admits a carr\'e du champ operator. In particular,~\eqref{eqn-ac} is true for all $f\in\mathcal{F}\cap L^\infty(X,\mu)$.
\end{corollary}
\begin{proof}
The proof that $\Gamma$ extends to represent all $f\in\mathcal{F}\cap L^\infty(X,\mu)$ is~\cite[Proposition 4.1.3]{BouleauHirsch}.
\end{proof}

The second corollary is of interest because it is known there are spaces that admit regular Dirichlet forms for which the energy measure $\nu_f$ is singular to $\mu$ for any non-constant $f\in\mathcal{F}\cap L^\infty(X,\mu)$, see~\cite{Kusuoka,BenBassatStrichartzTeplyaev}.  Examples of such spaces include the Sierpinski gasket, see for instance~\cite{BP, BarlowHambly97, Ka13, Kajino}. These spaces also have the property that $\DF(f,f)=0$ implies $f$ is constant, so the following result says that on these spaces $\mathbf{B}^{p,1/2}(X)$ consists of constant functions when $p>2$.

\begin{corollary}\label{singular Kusuoka}
Suppose that for all $f\in\mathcal{F}$ we have that $f$ is constant whenever $\mathcal{E}(f,f)=0$. Then, if $\DF$ is regular and the energy measure $\nu_f$ is singular to $\mu$ for any non-constant $f\in\mathcal{F}$, the space $\mathbf{B}^{p,1/2}(X)$ contains only constant functions when $p>2$.
\end{corollary}
\begin{proof}

Suppose that $f\in \mathbf{B}^{p,1/2}(X) $ and assume without loss of generality, by Remark \ref{rem-normalcontruction}, that $f\geqslant0$ is bounded. Then, from Lemma \ref{prop:convexityembed} $f^{p/2} \in \mathbf{B}^{2,1/2}(X)  =\mathcal F$. However, since $f$ is bounded, one also has  $f^{p/2} \in \mathbf{B}^{p,1/2}(X) $. Therefore, $f^{p/2} \in \mathbf{B}^{p,1/2}(X)\cap\mathcal F$. From the proof of Theorem~\ref{thm-ac} we can conclude that $\nu_{f^{p/2}}=0$, thus $f^{p/2}$ and hence $f$ are constants.   
\end{proof}

\subsection{Banach space property and reflexivity}

In this section we prove that  for $p\ge 1$ and $\alpha \ge 0$, $\mathbf{B}^{p,\alpha}(X)$ is always a Banach space which is moreover reflexive if $p>1$.

\begin{proposition}\label{prop-Banach}
 For $p\ge 1$ and $\alpha \ge 0$, $\mathbf{B}^{p,\alpha}(X)$ is a Banach space.
 \end{proposition}

\begin{proof}
Let $f_n$ be a Cauchy sequence in $\mathbf{B}^{p,\alpha}(X)$. Let $f$ be the $L^p$ limit of $f_n$. From 
Minkowski's inequality used from the representation \eqref{heat kernel measure} and conservativeness of $P_t$, one has 
\begin{align*}
 &  \left| \left(\int_X P_t (|f_n-f_n(y)|^p)(y) d\mu(y) \right)^{1/p} - \left(\int_X P_t (|f-f(y)|^p)(y) d\mu(y) \right)^{1/p} \right| \\
 & \ \ \ \ \ \ \ \  \ \ \ \ \le \left(\int_X P_t (|(f_n-f)-(f_n(y)-f(y))|^p)(y) d\mu(y) \right)^{1/p} \\
 &\ \ \  \ \ \ \ \ \ \ \ \ \le \left(\int_X P_t (|f_n-f|^p)(y) d\mu(y) \right)^{1/p}+\left(\int_X P_t (|f_n(y)-f(y)|^p)(y) d\mu(y) \right)^{1/p} \\
 &\ \ \  \ \ \ \ \ \ \  \ \  \le 2 \| f-f_n \|_{L^p(X,\mu)} .
\end{align*}
Therefore
\[
\lim_{n \to +\infty}  \left(\int_X P_t (|f_n-f_n(y)|^p)(y) d\mu(y) \right)^{1/p}=  \left(\int_X P_t (|f-f(y)|^p)(y) d\mu(y) \right)^{1/p},
\]
from which we deduce that
\begin{align*}
\frac{1}{t^\alpha}\left(\int_X P_t (|f-f(y)|^p)(y) d\mu(y) \right)^{1/p}
&=\lim_{n\to\infty}\frac{1}{t^\alpha}\left(\int_X P_t (|f_n-f_n(y)|^p)(y) d\mu(y) \right)^{1/p}\\
&\le \lim_{n\to\infty}\| f_n \|_{p,\alpha}<\infty.
\end{align*}
Therefore $f \in \mathbf{B}^{p,\alpha}(X)$ and $\| f \|_{p,\alpha} \le \lim_{n \to +\infty} \| f_n \|_{p,\alpha}$. Similarly,
for each fixed positive integer $m$,
\[
\| f-f_m \|_{p,\alpha} \le  \lim_{n \to +\infty} \| f_n -f_m \|_{p,\alpha}
\]
and taking the limit $m \to +\infty$ together with the fact that 
$(f_n)$ is Cauchy with respect to the seminorm $\|\cdot\|_{p,\alpha}$ completes the proof.
\end{proof}

We now turn to the reflexivity of $\mathbf{B}^{p,\alpha}(X)$. The Clarkson inequalities for $L^p$-functions are well-known. Given them, the following equivalent norm of $\|\cdot\|_{\mathbf{B}^{p,\alpha}(X)}$
immediately verifies the Clarkson inequalities for $\mathbf{B}^{p,\alpha}(X)$ given below. 
The equivalent norm, still denoted by $\| \cdot \|_{\mathbf{B}^{p,\alpha}(X)}$, is given by
\[
\| f \|_{\mathbf{B}^{p,\alpha}(X)} = \brak{\| f \|_{L^p(X,\mu)}^p +\| f \|_{p,\alpha}^p }^{\frac1p}.
\]

\begin{lemma}[Clarkson type inequalities]
Let $f,g\in \mathbf{B}^{p,\alpha}(X)$, $1<p<\infty$, and $q$ be the H\"older conjugate of $p$. If $2\le p<\infty$, then
\begin{equation}\label{eq:CTIge2}
\norm{(f+g)/2}_{\mathbf{B}^{p,\alpha}(X)}^p + \norm{(f-g)/2}_{\mathbf{B}^{p,\alpha}(X)}^p
\le \|f\|_{\mathbf{B}^{p,\alpha}(X)}^p/2 +\|g\|_{\mathbf{B}^{p,\alpha}(X)}^p/2.
\end{equation}
If $1<p\le 2$, then
\begin{equation}\label{eq:CTIle2}
\norm{(f+g)/2}_{\mathbf{B}^{p,\alpha}(X)}^{q} + \norm{(f-g)/2}_{\mathbf{B}^{p,\alpha}(X)}^{q}
\le \brak{\|f\|_{\mathbf{B}^{p,\alpha}(X)}^p/2 +\|g\|_{\mathbf{B}^{p,\alpha}(X)}^p/2}^{q-1}.
\end{equation}
\end{lemma}


By Proposition~\ref{prop-Banach} and by the discussion above, we know that $\mathbf B^{p,\alpha}(X)$ is a  Banach space.
By the above Clarkson inequalities, $\mathbf B^{p,\alpha}(X)$ is uniformly convex. These, together with the 
Milman-Pettis theorem, yield the following corollary.

\begin{corollary}\label{lem:reflexive}
For any $p > 1$ and $\alpha>0$, $\mathbf B^{p,\alpha}(X)$ is a  reflexive Banach space.
\end{corollary}

\subsection{Interpolation inequalities}

%

Now we turn our attention to interpolation inequalities. This exploration is in the spirit of the classical situation, where
it is known that the classical (metric) Besov classes of functions on Euclidean spaces are obtained by interpolation
between the Lebesgue spaces $L^p$ and Sobolev spaces $W^{1,p}$; see~\cite{GKS} for analogous results
in metric setting where the measure is doubling and supports a $p$-Poincar\'e inequality.
In our general setting, we have the following basic interpolation inequalities.

\begin{proposition}\label{interpolation inequality}
Let $\theta\in [0,1]$, $1 \le q,r <+\infty$ and $\beta, \gamma >0$. Let us assume $\frac{1}{p}=\frac{\theta}{q}+\frac{1-\theta}{r}$ and $\alpha=\theta \beta+(1-\theta)\gamma$. Then, $ \B^{q,\beta}(X)\cap \B^{r,\gamma}(X) \subset \B^{p,\alpha}(X) $  and for any $f\in \B^{q,\beta}(X)\cap \B^{r,\gamma}(X)$,
\[
\|f\|_{p,\alpha} \le \|f\|_{q,\beta}^{\theta} \|f\|_{r,\gamma}^{1-\theta}.
\]
\end{proposition}

\begin{proof}
Let $f \in \B^{q,\beta}(X)\cap \B^{r,\gamma}(X) $. One has for every $t >0$
\begin{align*}
    t^{-\alpha} \left( \int_X P_t (|f-f(y)|^p)(y) d\mu(y) \right)^{1/p}  = t^{-\theta \beta-(1-\theta)\gamma} \left( \int_X P_t (|f-f(y)|^p)(y) d\mu(y) \right)^{1/p}.
\end{align*}
Then, from H\"older's inequality
\begin{align*}
\int_X P_t (|f-f(y)|^p)(y) d\mu(y) &=\int_X P_t (|f-f(y)|^{p\theta+p(1-\theta)})(y) d\mu(y)  \\
 & \le \left( \int_X P_t (|f-f(y)|^q)(y) d\mu(y) \right)^{\frac{p \theta}{q}} \left( \int_X P_t (|f-f(y)|^r)(y) d\mu(y)\right)^{\frac{p (1-\theta)}{r}}.
\end{align*}
One deduces
\begin{align*}
  & t^{-\alpha} \left( \int_X P_t (|f-f(y)|^p)(y) d\mu(y) \right)^{1/p}  \\
 \le & t^{-\theta \beta}  \left( \int_X P_t (|f-f(y)|^q)(y) d\mu(y) \right)^{\frac{ \theta}{q}}  t^{-(1-\theta)\gamma} \left( \int_X P_t (|f-f(y)|^r)(y) d\mu(y)\right)^{\frac{ 1-\theta}{r}}.
\end{align*}
Taking the supremum over $t>0$ finishes the proof.
\end{proof}

\begin{remark}
This interpolation inequality opens the door to study the (real and complex) interpolation theory of our Besov spaces. In view of the previous interpolation inequalities, it would be natural to conjecture that $ (\B^{q,\beta}(X), \B^{r,\gamma}(X))_{\theta,p} = \B^{p,\alpha}(X) $, where $0<\theta<1$, $1< q,r <+\infty$ and $\alpha, \beta, \gamma, p$ are the same as in the above proposition.
\end{remark}

By Proposition~\ref{prop:energyasbesov} we know that $\B^{2,1/2}(X)=\mathcal F$. Therefore, 
by the above interpolation inequality from Proposition~\ref{interpolation inequality}, we have the following result.


\begin{corollary}\label{duality Dirichlet}
Let $1<p \le 2$ and $q$ be its conjugate, i.e. $\frac{1}{p}+\frac{1}{q}=1$. Let $0<\alpha<1$. Then, for any $f\in \mathcal F \cap \mathbf B^{p,\alpha}(X)$ and $g\in  \mathcal F \cap \mathbf B^{q,1-\alpha}(X)$, it holds that
\[
| \Ecal(f,g) | \le \|f\|_{p,\alpha} \|g\|_{q,1-\alpha}.
\]
\end{corollary}

\subsection{Pseudo-Poincar\'e inequalities and fractional powers of the generator}

Our goal in this section is to relate our Besov spaces to the domain of some fractional powers of the generator of the Dirichlet form.
In a very general framework, one can  resort to (Hille-Yosida) spectral theory to define the fractional powers of a closed operator 
$A$ on a Banach space $D(A)$ via the following formula
\[
(-A)^s f = \frac{\sin \pi s}{\pi} \int_0^\infty \lambda^{s-1} (\lambda I - A)^{-1} (-A)f\ d\lambda,
\]
for every $f\in D(A)$. In fact, using Bochner's subordination one can express the fractional powers of $A$ also in terms of the 
heat semi-group $P_t = e^{tA}$ via the following formula, see (5) in~\cite[page~260]{Yosida},
\begin{equation}\label{As}
(-A)^s f = - \frac{s}{\Gamma(1-s)} \int_0^\infty t^{-s-1} [P_t f - f]\ dt.
\end{equation}

With $A=L$ where
$L$ is the generator of $\mathcal{E}$, we set, for $0< s \le 1$, the class $\mathcal{L}_p^s$ to be the domain of the operator 
$(-L)^s$ in $L^p(X,\mu)$, $1 \le p <\infty$. In other words, $\mathcal{L}_p^s$ consists of functions from $L^p(X,\mu)$
for which there is a function $g\in L^p(X,\mu)$ such that $(-L)^s f=g$.

The following simple pseudo-Poincar\'e inequalities that are analogs of classical Sobolev embeddings, will later play a prominent role in Section \ref{section sobolev} and in  our three subsequent papers. In this section, we will use them to  prove that the fractional operator $(-L)^s: \B^{p,\alpha}(X) \to L^p(X,\mu)$ is bounded, where $L$ is the generator of the Dirichlet form $\mathcal{E}$ and $0<s <\alpha \le 1$.

\begin{lemma}[Pseudo-Poincar\'e inequalities]\label{pseudo-Poincare}
Let $ p \ge 1$ and $\alpha >0$. Then for every $f \in \mathbf{B}^{p,\alpha} (X)$, and $t \ge 0$,
\[
\| P_t f -f \|_{L^p(X,\mu)} \le t^\alpha \| f \|_{p,\alpha}.
\]
\end{lemma}

\begin{proof}
From conservativeness of the semigroup and H\"older's inequality of Lemma~\ref{lem:Pt-Holder}, we have
\begin{align*}
 \left(\int_X | P_t f (x)-f(x)|^p d\mu(x)\right)^{1/p} & = \left(\int_X | P_t (f -f(x))(x)|^p d\mu(x)\right)^{1/p} \\
 & \le \left( \int_X P_t (|f-f(x)|^p)(x) d\mu(x) \right)^{1/p}
 \le t^\alpha\, \| f \|_{p,\alpha}.\qedhere
\end{align*}
\end{proof}

\begin{remark}
Triebel \cite{Trie} (Section 1.13.6) introduced the interpolation spaces:
\[
(L^p(X,\mu), \mathcal{E} )_{\alpha,\infty}=\left\{ u \in L^p(X,\mu)\, :\, \sup_{t >0} t^{-\alpha} \| P_t u -u \|_{L^p(X,\mu )} <+\infty \right\}.
\]

From the previous lemma, it is therefore clear that $\mathbf{B}^{p,\alpha} (X) \subset (L^p(X,\mu), \mathcal{E})_{\alpha,\infty}.$ 
However, it may not be true that $\mathbf{B}^{p,\alpha} (X) = (L^p(X,\mu), \mathcal{E})_{\alpha,\infty}$, even when 
$X=\mathbb{R}^n$, see Remark 4.5 in \cite{MPPP} and \cite{Taibleson} (Theorems 4 and 4*).
\end{remark}

The following lemma will be useful:

\begin{lemma}
Let $L$ be the generator of $\mathcal{E}$, and let $p>1$, $0 <\alpha <1$. Then, there exists a constant 
$C>0$ such that for every $f \in \mathbf{B}^{p,\alpha}(X)$ and $t \ge 0$,
\[
\| LP_t f \|_{L^p(X,\mu)} \le C \frac{\| f \|_{p,\alpha}}{t^{1-\alpha}}.
\]
\end{lemma}

\begin{proof}
%
By the analyticity of the semigroup $P_t$,
see \eqref{analytic bound},  it follows that $\lim_{t \to +\infty} \| LP_t f \|_{L^p(X,\mu)} =0$ for $1<p<\infty$.
Then, we have by the semigroup property of $\{P_t\}_{t\in[0,\infty)}$ that
\begin{align*}
\| LP_{2t} f \|_{L^p(X,\mu)}  =\left\| \sum_{k=1}^{\infty}( LP_{2^k t} f -  LP_{2^{k-1} t} f )  \right \|_{L^p(X,\mu)} 
 & \le  \sum_{k=1}^{\infty} \left\| LP_{2^k t} f -  LP_{2^{k-1} t} f  \right \|_{L^p(X,\mu)} \\
 & \le \sum_{k=1}^{\infty} \left\| LP_{2^{k-1} t} (P_{2^{k-1}t} f -  f)  \right \|_{L^p(X,\mu)} \\
 &\le \sum_{k=1}^{\infty} \frac{1}{2^{k-1} t}  \left\| P_{2^{k-1}t} f -  f  \right \|_{L^p(X,\mu)} \\
 &\le C  \sum_{k=1}^{\infty} \frac{(2^{k-1} t)^{\alpha}}{2^{k-1} t}  \| f \|_{p,\alpha}  \\
 &\le C \frac{\| f \|_{p,\alpha}}{t^{1-\alpha}},
\end{align*}
where we used the analyticity of $P_t$ in the third inequality and the pseudo-Poincar\'e inequality in the fourth.
\end{proof}

One has then the following proposition:

\begin{proposition}
Let $\alpha \in (0,1]$, $p \ge 1$ and $0<s<\alpha$. Then
\[
\B^{p,\alpha}(X) \subset \mathcal{L}_p^s,
\]
and there exists a constant $C=C_{s,\alpha}$ such that for every $f\in \B^{p,\alpha}(X)$,
\begin{equation}\label{eq:Young}
\| (-L)^s f \|_{L^p(X,\mu)} \le C \|  f \|^{1-\frac{s}{\alpha}}_{L^p(X,\mu)} \| f \|_{p,\alpha}^{\frac{s}{\alpha}}.
\end{equation}
In particular, $(-L)^s: \B^{p,\alpha}(X) \to L^p(X,\mu)$ is bounded.
\end{proposition}

\begin{proof}
Let  $f \in \B^{p,\alpha}(X)$. We need to prove that the integral $x\mapsto \int_0^\infty t^{-s-1} (P_t f(x) - f(x))\ dt$ is finite
for almost every $x\in X$, 
and therefore that $f \in \mathcal{L}_p^s$. For $\delta >0$, one has
\begin{align*}
 \left\| \int_0^\infty t^{-s-1} (P_t f - f)\ dt \right\|_{L^p(X,\mu)} & \le  \int_0^\infty t^{-s-1}  \| P_t f - f \|_{L^p(X,\mu)}  dt \\
  & \le \int_0^\delta t^{-s-1}  \| P_t f - f \|_{L^p(X,\mu)}  dt +  \int_\delta^\infty t^{-s-1}  \| P_t f - f \|_{L^p(X,\mu)}  dt  \\
  & \le \| f \|_{p,\alpha} \int_0^\delta t^{-s-1 +\alpha}   dt +2 \| f \|_{L^p(X,\mu)} \int_\delta^\infty t^{-s-1}    dt \\
  & \le \| f \|_{p,\alpha} \frac{\delta^{\alpha-s}}{\alpha -s}+2 \| f \|_{L^p(X,\mu)} \frac{\delta^{-s}}{s}.
\end{align*}
Choosing $\delta=1$ in the above shows the boundedness of $(-L)^s$.
To see~\eqref{eq:Young}, we choose $\delta>0$ that satisfies
\[
\delta^\alpha=2\frac{\|f\|_{L^p(X,\mu)}}{\|f\|_{p,\alpha}}\, \frac{\alpha-s}{s}
\]
so that 
\[
\| f \|_{p,\alpha} \frac{\delta^{\alpha-s}}{\alpha -s}=2 \| f \|_{L^p(X,\mu)} \frac{\delta^{-s}}{s}.
\]
Then
\begin{align*}
\frac{\Gamma(1-s)}{s}\| (-L)^s f \|_{L^p(X,\mu)}&=\left\| \int_0^\infty t^{-s-1} (P_t f - f)\ dt \right\|_{L^p(X,\mu)}\\
&\le 2 \| f \|_{L^p(X,\mu)} \frac{\delta^{-s}}{s}\\
&=\frac{2^{2-s/\alpha}}{s^{1-s/\alpha}(\alpha-s)^{s/\alpha}}\, \|f\|_{L^p(X,\mu)}^{1-s/\alpha}\, \|f\|_{p,\alpha}^{s/\alpha}.\qedhere
\end{align*}
\end{proof}



\section{Continuity of $P_t$ on the Besov spaces and critical exponents}

Our goal in this section is to study the continuity properties of the semigroup $P_t$ in the Besov spaces $\B^{p,\alpha}(X)$ with range $1<p \le 2$ and parameter $\alpha=\frac{1}{2}$. As corollaries we will deduce several important properties of the Besov spaces themselves. In particular, we will obtain the non-trivial fact that for $1 <  p \le 2$, the Besov space $\B^{p,1/2}(X)$ contains the $L^p(X,\mu)$ domain of $L$.

As before, throughout the section, let $(X,\mu,\mathcal{E},\mathcal{F})$ be a Dirichlet space and let $\{P_{t}\}_{t\in[0,\infty)}$ denote the associated  heat semigroup.

\subsection{Continuity}

The main result of this section quantifies a regularization property of the heat semigroup as follows.

\begin{theorem}\label{continuity Besov chapter 1}
Let $1<p\le 2$. There exists a constant $C_p>0$ such that for every $f \in L^p(X,\mu)$ and $t \ge 0$
\[
\| P_t f \|_{p,1/2} \le \frac{C_p}{t^{1/2}} \| f \|_{L^p(X,\mu)}.
\]
In particular $P_t: L^p(X,\mu) \to \B^{p,1/2}(X)$ is bounded for $t>0$.
\end{theorem}

It is remarkable that Theorem \ref{continuity Besov chapter 1} applies to any Dirichlet space $(X,\mu, \mathcal{E},\mathcal{F})$.  We will see in \cite{ABCRST2, ABCRST3} that the study of the continuity of the semigroup in the Besov spaces $\B^{p,\alpha}(X)$ with range $p > 2$ requires additional assumptions on the space (weak Bakry-\'Emery type curvature condition).

To prove this theorem we need the following auxiliary result. The proof of this auxiliary result is obtained 
from some deep ideas originally due to Nick Dungey~\cite{Dungey} and developed further by Li Chen in~\cite{Chen}.

\begin{lemma}\label{Lemma interpolation}
Let $1<p \le 2$. There exists a constant $C_p>0$ such that for every non-negative $f \in L^p(X,\mu)$ and $t >0$
\[
 \left( \int_X P_t (|f-f(y)|^p)(y) d\mu(y) \right)^{1/p} \le C_p \| f \|^{1/2}_{L^p(X,\mu)} \| P_t f -f \|^{1/2}_{L^p(X,\mu)}.
\]
\end{lemma}

\begin{proof}
Let $1 <p \le 2$ and $t>0$ be fixed in the following proof. The constant $C$ in the following will denote a positive constant depending only on $p$ that may change from line to line. 
For $\alpha,\beta \ge 0$, set 
\[
\gamma_p(\alpha,\beta):=p\alpha(\alpha-\beta)-\alpha^{2-p} (\alpha^p-\beta^p)
\]
and for a non-negative function $f \in L^p(X,\mu)$
\begin{align*}
\Gamma_p(f)(x) 
&:=p f(x) \int_X (f(x)-f(y))\, p_t(x,dy) - f^{2-p}(x) \int_X \brak{f^p(x)-f^p(y)} \,p_t(x,dy)
\\&= \int_X \gamma_p(f(x),f(y))\, p_t(x,dy).
\end{align*}
Note that from~\cite[Lemma~3.5]{LiChen}, one has for any $\alpha,\beta \ge 0$ 
\[
(p-1) (\alpha -\beta)^2 \le \gamma_p(\alpha,\beta)+\gamma_p(\beta,\alpha) \le p (\alpha-\beta)^2
\]
and that, similarly to \cite[page 122]{Dungey}, one has $\Gamma_p(f) \ge 0$.
Then the same argument as in \cite[Lemma~3.6]{LiChen} gives
\begin{align*}
\lefteqn{\int_X\int_X  |f(x)-f(y)|^p\, p_t(x,dy) \,d\mu(x)} \quad&\\
&\le 
 C \int_X\int_X \brak{\gamma_p(f(x),f(y))+\gamma_p(f(y),f(x))}^{p/2} p_t(x,dy)\,d\mu(x)
\\&\le 
C \int_X\int_X \brak{\gamma_p^{p/2}(f(x),f(y))+\gamma_p^{p/2}(f(y),f(x))} p_t(x,dy) \,d\mu(x)
\\&= 
C\int_X\int_X \gamma_p^{p/2}(f(x),f(y))\, p_t(x,dy)\, d\mu(x)
\\&\le
C\int_X\brak{\int_X  \gamma_p(f(x),f(y))\, p_t(x,dy)}^{p/2} d\mu(x)
\\&= 
C\int_X \Gamma_p^{p/2}(f)(x)\, d\mu(x).
\end{align*}
Here the fourth line follows from the symmetry property of heat kernel measure in \eqref{heat kernel measure symmetry}.
Denote $\Delta_t=I -P_t$. Then $\Delta_t$ is the generator of a strongly continuous semigroup $\{e^{-s\Delta_t}\}_{s\in[0,\infty)}$ on $L^p(X,\mu)$ given by $e^{-s\Delta_t} =\sum_{n=0}^{\infty} \frac{s^n}{n!} (P_t-I)^n$. We then follow the proof of Theorem~1 in \cite{LiChen} (see also Theorem 1.3 in  \cite{Dungey}) by taking $u(s,x)=e^{-s\Delta_t} f (x)$. Note that
\begin{align*}
\Gamma_p(u)& =pu (  u - P_t u) -u^{2-p} (u^p -P_t (u^p)) \\
 & =pu \Delta_t u  -u^{2-p} \Delta_t (u^p) \\
 &=-pu \partial_s u  -u^{2-p} \Delta_t (u^p) \\
 & =-u^{2-p} \left( \partial_s +\Delta_t \right)u^p.
\end{align*}
Set now
\[
J(s,x)=- \left( \partial_s +\Delta_t \right)u^p(s,x),
\]
so that
\[
\Gamma_p(u)=u^{2-p} J.
\]
Note that since $u \ge 0$ and $\Gamma_p(u) \ge 0$, one has $J \ge 0$.
One has then from H\"older's inequality
\begin{align*}
\int_X \Gamma_p^{p/2}(u) d\mu & =\int_X u^{p(2-p)/2} J^{p/2} d\mu \\
 & \le \left( \int_X u^p d\mu \right)^{\frac{2-p}{2}} \left( \int_X J d\mu \right)^{p/2}.
\end{align*}
Observe that $u\in L^p(X,\mu)$ and hence $u^p\in L^1(X,\mu)$. Then $\int P_t (u^p)d\mu=\int u^p P_t1d\mu=\int u^pd\mu$ by symmetry and the conservative property of $P_t$.  It follows that  $\int_X \Delta_t u^p d\mu=0$.
One computes then
\[
\int_X J d\mu =-\int_X \left( \partial_s +\Delta_t \right)u^p(s,x) d\mu =- \int_X \partial_s (u^p) d\mu=-p \int_X u^{p-1} \partial_s u d\mu=p \int_X u^{p-1} \Delta_t u d\mu.
\]
Thus, we have from H\"older's inequality
\[
\int_X J d\mu \le p \| u \|_{L^p(X,\mu)}^{p-1}  \| \Delta_t u \|_{L^p(X,\mu)} .
\]
From the definition of $\Delta_t$ one concludes therefore
\[
\left( \int_X\int_X  |u(s,x)-u(s,y)|^p \,p_t(x,dy)\, d\mu(x) \right)^{1/p} \le C \| u(s,\cdot) \|^{1/2}_{L^p(X,\mu)} \| P_t u(s,\cdot) -u(s,\cdot) \|^{1/2}_{L^p(X,\mu)}.
\]
Letting $s \to 0^+$ yields 
\begin{equation*}
\left( \int_X\int_X |f(x)-f(y)|^p \,p_t(x,dy)\, d\mu(x) \right)^{1/p} \le C \| f \|^{1/2}_{L^p(X,\mu)} \| P_t f -f \|^{1/2}_{L^p(X,\mu)}.\qedhere
\end{equation*}
\end{proof}

We are now ready to prove Theorem \ref{continuity Besov chapter 1}.

\begin{proof}[\textbf{Proof of Theorem \ref{continuity Besov chapter 1}}]
Let $f \in L^p(X,\mu)$. We can assume $f \ge 0$. If not, it is enough to decompose $f$ as $f^+-f^-$ with $f^+=\max \{ f ,0 \}$ and $f^-=\max \{ -f ,0 \}$. Let $s,t >0$, applying Lemma \ref{Lemma interpolation} to $P_s f$, one obtains
\[
 \left( \int_X P_t (|P_sf-P_sf(y)|^p)(y) d\mu(y) \right)^{1/p} \le C_p \| P_s f \|^{1/2}_{L^p(X,\mu)} \| P_{t+s} f -P_sf \|^{1/2}_{L^p(X,\mu)}.
\]
Note that $\| P_s f \|_{L^p(X,\mu)} \le \|  f \|_{L^p(X,\mu)}$ and that
\begin{align*}
\| P_{t+s} f -P_sf \|_{L^p(X,\mu)} & = \left\| \int_0^t LP_{s+u} f  du \right\|_{L^p(X,\mu)} \\
 & =\left\| \int_0^t P_u LP_{s} f  du \right\|_{L^p(X,\mu)}  \\
 &  \le \int_0^t \left\|  P_u LP_{s} f \right\|_{L^p(X,\mu)} du \\
 & \le t  \left\|   LP_{s} f \right\|_{L^p(X,\mu)}  \\
 & \le C \frac{t}{s}  \left\|    f \right\|_{L^p(X,\mu)},
\end{align*}
where in the last step we used analyticity of the semigroup. One concludes
\[
 \left( \int_X P_t (|P_sf-P_sf(y)|^p)(y) d\mu(y) \right)^{1/p} \le C \left(  \frac{t}{s} \right)^{1/2} \|  f \|_{L^p(X,\mu)}.
\]
Dividing both sides by $\sqrt{t}$ and taking the supremum over $t>0$ complete the proof.
\end{proof}

We now collect several corollaries of Theorem \ref{continuity Besov chapter 1}. The following surprising result shows that when $p \ge 2$,  the quantity $$\sup_{t >0} t^{-1/2} \| P_t f -f \|_{L^p(X,\mu)}$$ can actually always  be controlled by   $$\liminf_{t \to 0^+}  t^{-1/2} \left( \int_X P_t (|f-f(y)|^p)(y) d\mu(y) \right)^{1/p}.$$
This is another manifestation of the locality in time property of our Besov spaces (see also Section \ref{locality in time section}).

\begin{proposition}\label{Critical bound Chapter 1}
Let $ 2 \le p <+\infty  $. For every $f \in L^p(X,\mu)$, and $t \ge 0$,
\[
\| P_t f -f \|_{L^p(X,\mu)} \le C_p   t^{1/2}  \liminf_{s \to 0}  s^{-1/2} \left( \int_X P_s (|f-f(y)|^p)(y) d\mu(y) \right)^{1/p}.
\]
\end{proposition}

\begin{proof}
For  $\tau \in(0,\infty)$  we set
\begin{equation}\label{eq:energy-pt-Lp3}
\mathcal{E}_{\tau}(u,v):= \frac{1}{ \tau} \int_X (P_\tau -I )u\,  v d\mu.
\end{equation}
Let $f \in  L^p(X,\mu)$ and $g \in L^q(X,\mu)$ where $q$ is the conjugate exponent of $p$. We note that, 
\begin{align*}
\int_0^t \mathcal{E}_\tau (P_s f ,g) ds& = \int_0^t \frac{1}{ \tau} \int_X   (P_{s+\tau} f -P_{s} f) g  d\mu ds \\
 &= \int_X \left( \frac{1}{\tau} \int_t^{t+\tau} P_s f ds  - \frac{1}{\tau} \int_0^{\tau} P_s f ds \right) g d\mu
\end{align*}
Therefore, using the strong continuity of the semigroup in $L^p(X,\mu)$, one has for $t \ge 0$,
\[
\int_X (P_t f -f ) g d\mu = \lim_{\tau \to 0^+} \int_0^t \mathcal{E}_\tau (P_s f ,g) ds  .
\]
Note now that $\mathcal{E}_{\tau}(P_s f ,g)=\mathcal{E}_{\tau}( f,P_s g)$ and that from H\"older's inequality (applied as in the proof of Proposition \ref{interpolation inequality}) 
\begin{align*}
2 | \mathcal{E}_{\tau}(f,P_s g) | & \le \tau^{-1/2} \left( \int_X P_\tau  (|P_sg-P_sg(y)|^q)(y) d\mu(y) \right)^{1/q} \tau^{-1/2} \left( \int_X P_\tau (|f-f(y)|^p)(y) d\mu(y) \right)^{1/p} \\
 & \le  \tau^{-1/2} \left( \int_X P_\tau (|f-f(y)|^p)(y) d\mu(y) \right)^{1/p} \| P_s g \|_{q,1/2} \\
 &\le C_p  \tau^{-1/2} \left( \int_X P_\tau (|f-f(y)|^p)(y) d\mu(y) \right)^{1/p}  s^{-1/2} \| g \|_{L^q(X,\mu)}.
\end{align*}
One has therefore
\begin{align*}
\left| \int_X (P_t f -f ) g d\mu \right| \le  C_p t^{1/2} \| g \|_{L^q(X,\mu)} \liminf_{s \to 0}  s^{-1/2} \left( \int_X P_s (|f-f(y)|^p)(y) d\mu(y) \right)^{1/p},
\end{align*}
and we conclude by $L^p-L^q$ duality.
\end{proof}
One deduces:
\begin{corollary}\label{corollary 1.25}
Let $ 2 \le p <+\infty  $ and $\alpha >1/2$. If $ f \in \B^{p,\alpha} (X)$ then $\mathcal{E}(f,f)=0$.
\end{corollary}

\begin{proof}
Indeed, for $ f \in \B^{p,\alpha} (X)$ with $\alpha >1/2$ one has
\[
 \liminf_{s \to 0}  s^{-1/2} \left( \int_X P_s (|f-f(y)|^p)(y) d\mu(y) \right)^{1/p}=0,
 \]
 so that for every $t \ge 0$, $P_t f=f$, and thus $\mathcal{E}(f,f)=0$.
\end{proof}

Our final corollary of Theorem \ref{continuity Besov chapter 1} is as follows.

\begin{prop}
Let $1<p \le 2$. Let $L$ be the generator of $\mathcal{E}$ and $\mathcal{L}_p$ be the domain of $L$ in $L^p(X,\mu)$.  Then
\[
\mathcal{L}_p \subset \B^{p,1/2}(X)
\]
and for every $f \in \mathcal{L}_p$,
\begin{equation}\label{eq:multi2}
\|f\|^2_{p,1/2} \le C \norm{ Lf}_{L^p(X,\mu)} \| f\|_{L^p(X,\mu)}.
\end{equation}
\end{prop}
\begin{proof}
Write for $\lambda >0$
\[
R_{\lambda}f=(L-\lambda)^{-1} f=\int_0^{\infty} e^{-\lambda t} P_tf dt. 
\]
Consequently
\[
\|R_{\lambda}f\|_{p,1/2} \le \int_0^{\infty} e^{-\lambda t} \|P_tf\|_{p,1/2}  dt 
\le \int_0^{\infty} e^{-\lambda t}  \frac{C}{t^{1/2}} \|f\|_{L^p(X,\mu)}  dt\le C \lambda^{-1/2} \|f\|_{L^p(X,\mu)}.
\]
It follows that 
\[
\|f\|_{p,1/2} \le C  \lambda^{-1/2} \norm{(L-\lambda)f}_p  \le C ( \lambda^{-1/2} \|Lf\|_{L^p(X,\mu)}+ \lambda^{1/2} \|f\|_{L^p(X,\mu)}).
\]
Taking $\lambda=\|Lf\|_{L^p(X,\mu)} \|f\|_{L^p(X,\mu)}^{\,-1}$ gives the result.
\end{proof}

\subsection{Critical Besov exponents}\label{section critical}

One can summarize several of our findings about the density or the triviality of our spaces $\mathbf{B}^{p,\alpha}(X)$ by introducing the notion of Besov critical exponents. 
Let $p \ge 1$.  For the space $X$ we define the $L^p$ Besov density critical exponent $\alpha_p^*(X)$ and triviality critical exponent $\alpha_p^\#(X)$ as follows:
\begin{equation*}
\alpha_p^*(X) = \sup \{ \alpha>0\,:\, \mathbf{B}^{p,\alpha}(X) \text{ is dense in } L^p(X,\mu)\}.
\end{equation*}
\[
\alpha^\#_p(X)=\sup \{ \alpha >0\, :\, \mathbf{B}^{p,\alpha}(X) \text{ contains non-constant functions} \}.
\]
Evidently $\alpha_p^*(X)\leq\alpha_p^\#(X)$. Critical Besov exponents of this and similar types have appeared in several previous works~\cite{GHL:TAMS2003,MR2743439,GuLau}.  In particular, Grigor'yan~\cite{MR2743439} points out that when  Theorem~\ref{sub gaussian intro} can be applied, we know $\mathbf{B}^{p,\alpha}(X)$ can be defined in a purely metric fashion and therefore the critical exponents are determined by the metric-measure structure of $X$ and are independent of any heat kernel. He also proves the exponent $\alpha_2^*(X)=\frac{1}{2}$ if $P_t$ is stochastically complete, see also Proposition~\ref{Besov critical exponents}(4) below.
There does not seem to be any literature on  whether  $\alpha^*(X)$ and $\alpha^\#(X)$ are distinct, but note that Gu and Lau~\cite{GuLau} gave examples of spaces and Dirichlet forms for which the Besov critical exponent  for density of $\mathbf{B}^{2,\alpha}(X)$ in $C(X)$ is strictly less than $\alpha_2^\#(X)$.

\begin{proposition}\label{Besov critical exponents}
The following are true:
\begin{enumerate}
\item  Both $p\mapsto \alpha_p^* (X)$ and $p\mapsto \alpha_p^\# (X)$ are non-increasing;
\item For $1\leq p \leq 2$ we have $\alpha_p^\# (X)\geq \alpha_p^*(X) \geq \frac{1}{2}$.
\end{enumerate}
If we assume  that $\mathcal{E}(f,f)=0$ implies $f$ constant, then we have in addition
\begin{enumerate}
\setcounter{enumi}{2}
\item If $1\leq p\leq 2$ then $\alpha_p^*(X)\leq\alpha_p^\#(X)\leq \frac1p$;
\item $\alpha^*_2(X)=\alpha_2^\#(X)=\frac12$;
\end{enumerate}
\begin{enumerate}
\setcounter{enumi}{4}
\item For $2\leq p<\infty$ one has $\alpha_p^*(X)\leq\alpha_p^\#(X)\leq \frac12$.
\end{enumerate}
Furthermore if $\DF$ is regular and the energy measure $\nu_f$ for each non-constant $f\in\mathcal{F}$ is singular to $\mu$ (as is the case on some fractals) we obtain
\begin{enumerate}
\setcounter{enumi}{5}
\item For $p>2$ one has $\alpha_p^*(X)\leq\alpha_p^\#(X)< \frac12$.
 \end{enumerate} 
\end{proposition}

\begin{proof}\ 

\begin{enumerate}
\item This is a direct application of Proposition~\ref{prop:convexityembed} and Corollary~\ref{cor:densitybelow2}.
\item We proved in  Proposition~\ref{prop:energyasbesov} that $\mathbf{B}^{2,1/2}(X)=\mathcal{F}$ which is dense in $L^2$, which proves the result for $p=2$.  For $1\leq p\leq 2$ we use Claim~1.
\item This follows from  Proposition \ref{prop:BesovCEp}.
\item Combine Claim~2 and Claim~3.
\item This follows from Corollary~\ref{corollary 1.25}.
\item This is Corollary~\ref{singular Kusuoka}.\qedhere
\end{enumerate}
\end{proof}

We now  present some conjectures on the critical exponents. We state them for $\alpha_p^\#(X)$, but similar results would be expected to hold for $\alpha_p^*(X)$.

 \begin{remark}\label{remark conjecture 1}
In view of the duality given by Corollary \ref{duality Dirichlet}, it is natural to conjecture that under suitable conditions one may have
\[
\alpha^\#_p(X)+\alpha_q^\#(X)=1
\]
if $p$ and $q$ are conjugates, i.e. satisfy $\frac{1}{p}+\frac{1}{q}=1$. 
\end{remark}

\begin{remark}
We will see in \cite{ABCRST2,ABCRST3} that for local Dirichlet forms the limit $$\alpha^\#_\infty(X)=\lim_{p \to +\infty} \alpha^\#_p(X)$$ is closely related to a H\"older regularity property in space of the heat semigroup. If the conjecture in Remark \ref{remark conjecture 1} is true, then classical interpolation theory (see Proposition \ref{interpolation inequality}) suggests that it is reasonable to expect that for every $p \ge 1$:
\[
\alpha^\#_p (X) =\frac{1}{p} +\left( 1-\frac{2}{p} \right) \alpha^\#_\infty(X).
\]
 \end{remark}
 
 \begin{example}
 For strictly local Dirichlet  forms with absolutely continuous energy measures, we will see in \cite{ABCRST2} that one generically has $\alpha^\#_p(X)=\alpha_p^*(X)=\frac{1}{2}$ for every $p \ge 1$.
 \end{example}
 
 \begin{example}
On the infinite Sierpinski gaskets    $\alpha_1^\#(X)=\frac{d_H}{d_W}$, where $d_H$ is the Hausdorff dimension of $X$ and $d_W$ its walk dimension, see \cite{ABCRST3}. Finding the exact value of $\alpha_1^\#(X)$ is, in general,  an open question; in particular it is open for Sierpinski carpets. 
 \end{example}
 
 \section{Sobolev and isoperimetric  inequalities }\label{section sobolev}

In this section, we are interested in sharp Sobolev type embeddings (the case $p=1$ corresponds to isoperimetric type results) for the Besov spaces studied in this paper. 
 
 Let $(X,\mu,\mathcal{E},\mathcal{F})$ be a Dirichlet space. Let $\{P_{t}\}_{t\in[0,\infty)}$ denote the Markovian semigroup associated with
$(X,\mu,\mathcal{E},\mathcal{F})$. Throughout the section, we shall assume that $P_t$ admits a measurable heat kernel $p_t(x,y)$ satisfying, for some
$C>0$ and $\beta >0$,
\begin{equation}\label{eq:subGauss-upper3}
p_{t}(x,y)\leq C t^{-\beta}
\end{equation}
for $\mu\times\mu$-a.e.\ $(x,y)\in X\times X$, and for each $t\in\bigl(0,+\infty \bigr)$. 
Our goal in this section is to prove  for the space $\mathbf{B}^{p,\alpha}(X)$ global Sobolev embeddings with sharp exponents and one of the main results will be the following weak-type Sobolev inequality and the corresponding isoperimetric inequality:

\begin{theorem}\label{polintro2}
Let $0<\alpha <\beta $. Let $1 \le p < \frac{\beta}{\alpha} $. There exists a constant $C_{p,\alpha} >0$ such that for every $f \in \mathbf{B}^{p,\alpha}(X) $,
\[
\sup_{s \ge 0}\, s\, \mu \left( \{ x \in X\, :\, | f(x) | \ge s \} \right)^{\frac{1}{q}} \le C_{p,\alpha} \| f \|_{p,\alpha},
\]
where $q=\frac{p\beta}{ \beta -p \alpha}$. Therefore, there exists a constant $C_{\emph{iso}} >0$, 
given in 
Proposition
\ref{prop:iso}, such that for every
subset  $E\subset X$ with $\mathbf{1}_E \in \mathbf{B}^{1,\alpha}(X) $
\[
\mu(E)^{\frac{\beta-\alpha}{\beta}} \le C_{\emph{iso}} \| \mathbf{1}_E \|_{1,\alpha}.
\]

\end{theorem}

\subsection{Weak type Sobolev inequality}\label{S:Weak_Sobolev}

We follow and adapt to our setting a general approach to Sobolev inequalities developed in \cite{BCLS} (see also  \cite{saloff2002}). The pseudo-Poincar\'e inequality proved in Lemma \ref{pseudo-Poincare} plays a fundamental role here.

\begin{lemma}\label{weak type}\label{gagliardo}
Let $1 \le p,q <+\infty$ and $\alpha >0$. There exists a constant $C_{q,\alpha,\beta}>0$ such that for every $f \in \mathbf{B}^{p,\alpha}(X) \cap L^q(X,\mu)$ and $s \ge 0$,
\[
\sup_{s \ge 0} s^{1 +q\frac{\alpha}{\beta}} \mu \left( \{ x \in X\, :\, | f(x) | > s \} \right)^{\frac{1}{p}}\le C_{q,\alpha,\beta} \| f \|_{p,\alpha} \| f \|_{L^q(X,\mu)}^{q\frac{\alpha}{\beta}}.
\]
\end{lemma}

\begin{proof}
We adapt an argument given in the proof of Theorem 9.1 in \cite{BCLS}.
Let $f \in \mathbf{B}^{p,\alpha}(X) $ and denote
\[
F(s)=\mu \left( \{ x \in X\, :\, | f(x) | > s \} \right).
\]
We have then
\[
F(s) \le \mu \left( \{ x \in X\, :\, | f(x) -P_t f (x) | > s/2 \} \right)+\mu \left( \{ x \in X\, :\, | P_t f (x) | > s/2 \} \right).
\]
Now, from the heat kernel upper bound $p_t(x,y) \le C t^{-\beta}$,  $ t>0$,
one deduces, for $g \in L^1(X,\mu)$, that $|P_t g (x) | \le Ct^{-\beta} \| g \|_{L^1(X,\mu)}.$
Since $P_t$ is a contraction in $L^\infty(X,\mu)$, by the Riesz-Thorin interpolation one obtains
\[
| P_t f (x) | \le \frac{C^{1/q}}{t^{\beta /q}} \| f \|_{L^q(X,\mu)}.
\]
Therefore, for $s= 2 \frac{C^{1/q}}{t^{\frac{\beta}{ q}}} \| f \|_{L^q(X,\mu)}$, one has $\mu \left( \{ x \in X\, :\, | P_t f (x) | > s/2 \} \right)=0.$
On the other hand, from  Theorem~\ref{pseudo-Poincare},
\[
\mu \left( \{ x \in X\, :\, | f(x) -P_t f (x) | > s/2 \} \right) \le 2^p s^{-p} t^{p\alpha} \| f \|^{p}_{p,\alpha}.
\]
We conclude that
\begin{equation*}
F(s)^{1/p} \le \widetilde{C} s^{-1 -q\frac{\alpha}{\beta}} \| f \|_{\alpha,p} \| f \|_{L^q(X,\mu)}^{q\frac{\alpha}{\beta}}.\qedhere
\end{equation*}
\end{proof}

As a corollary, we are now ready to prove the weak Sobolev inequality.

\begin{theorem}\label{pol}
Let $0<\alpha <\beta $. Let $1 \le p < \frac{\beta}{\alpha} $. There exists a constant $C_{p,\alpha,\beta} >0$ such 
that for every $f \in \mathbf{B}^{p,\alpha}(X) $,
\[
\sup_{s \ge 0}\, s\, \mu \left( \{ x \in X\, :\, | f(x) | \ge s \} \right)^{\frac{1}{q}} \le C_{p,\alpha,\beta} \| f \|_{p,\alpha},
\]
where $q=\frac{p\beta}{ \beta -p \alpha}$.
\end{theorem}

\begin{proof}
Let $f \in \mathbf{B}^{p,\alpha}(X) $ be a non-negative function. For $k \in \mathbb{Z}$, we denote
\[
f_k=(f-2^k)_+ \wedge 2^k.
\]
Observe that $f_k \in L^p(X,\mu)$ and $\| f_k \|_{L^p(X,\mu)} \le \| f \|_{L^p(X,\mu)}$. Moreover, for every $x,y \in X$, 
$|f_k (x)-f_k(y)|\le |f(x)-f(y)|$ and so $\| f_k\|_{p,\alpha} \le \| f\|_{p,\alpha}$. We also note that $f_k \in L^1(X,\mu)$, with
\[
\| f_k \|_{L^1(X,\mu)} =\int_X |f_k| d\mu \le 2^k \mu (\{ x \in X\, :\, f(x) \ge 2^k \}).
\]
We now use Lemma \ref{gagliardo} to deduce:
\begin{align*}
\sup_{s \ge 0} s^{1 +\frac{\alpha}{\beta}} \mu \left( \{ x \in X\, :\, f_k(x) > s \} \right)^{\frac{1}{p}}
 & \le C_{\alpha,\beta} \| f_k \|_{p,\alpha} \| f_k \|_{L^1(X,\mu)}^{\frac{\alpha}{\beta}} \\
 & \le C_{\alpha,\beta}\| f_k \|_{p,\alpha} \left( 2^k \mu (\{ x \in X\, :\, f(x) \ge 2^k \})\right)^{\frac{\alpha}{\beta}}.
\end{align*}
In particular, by choosing $s=2^k$ we obtain
\[
2^{k \left(1 +\frac{\alpha}{\beta} \right)} \mu \left( \{ x \in X\, :\, f(x) \ge 2^{k+1} \} \right)^{\frac{1}{p}}
 \le C_{\alpha,\beta} \| f_k \|_{p,\alpha} \left( 2^k \mu (\{ x \in X\, :\, f(x) \ge 2^k \})\right)^{\frac{\alpha}{\beta}}.
\]
Let
\[
M(f)=\sup_{k \in \mathbb{Z}} 2^k \mu (\{ x \in X\, :\, f(x) \ge 2^k \})^{1/q},
\]
where $q=\frac{p\beta}{ \beta-p \alpha}$. Using the fact that $\frac{1}{q}=\frac{1}{p}-\frac{\alpha}{\beta}$ and the 
previous inequality we obtain:
\[
2^{k} \mu \left( \{ x \in X\, :\, f(x) \ge 2^{k+1} \} \right)^{\frac{1}{p}}
 \le 2^{ -\frac{kq\alpha}{\beta}} C_{\alpha,\beta}\| f \|_{p,\alpha} M(f)^{\frac{q\alpha}{\beta}}.
\]
and
\[
2^k \mu \left( \{ x \in X\, :\, f(x) \ge 2^{k+1} \} \right)^{\frac{1}{q}}
 \le C^{\frac{p}{q} }_{\alpha,\beta} \| f \|_{p,\alpha}^{p/q}M(f)^{\frac{p\alpha}{\beta}}.
\]
Therefore
\[
M(f)^{1-\frac{p\alpha}{\beta}} \le 2 C^{\frac{p}{q} }_{\alpha,\beta} \| f \|_{p,\alpha} ^{p/q}.
\]
One concludes
\[
M(f) \le 2^{q/p} C_{\alpha,\beta} \| f \|_{p,\alpha} .
\]
This easily yields:
\[
\sup_{s \ge 0} s \mu \left( \{ x \in X\, :\, f(x) \ge s \} \right)^{\frac{1}{q}} \le 2^{1+q/p} C_{\alpha,\beta} \| f \|_{p,\alpha} .
\]
Let now $f \in \mathbf{B}^{p,\alpha}(X) $, which is not necessarily non-negative. From the previous inequality applied to $|f|$, we deduce
\begin{equation*} 
\sup_{s \ge 0}\, s\, \mu \left( \{ x \in X\, :\, |f(x) | \ge s \} \right)^{\frac{1}{q}} 
\le 2^{1+q/p} C_{\alpha,\beta} \| | f | \|_{p,\alpha} 
  \le 2^{1+q/p} C_{\alpha,\beta} \| f \|_{p,\alpha}.\qedhere
\end{equation*} 
\end{proof}

\subsection{Capacitary estimates}

It is well-known that Sobolev inequalities are related to capacitary estimates, see for instance \cite[Section 10]{BCLS}. In the current subsection, we explore this relation in our case. Let $p \ge 1$ and $0< \alpha <\beta$. For a measurable set $A \subset X$, we define its $(\alpha,p)$-capacity:
\[
\mathbf{Cap}^\alpha_p (A)=\inf \{ \| f \|_{\alpha,p}^p\, :\, f \in \mathbf{B}^{\alpha,p}(X), \mathbf{1}_A \le f \le 1 \}.
\]

We have the following corollary:

\begin{corollary}
Let $0<\alpha <\beta $. Let $1 \le p < \frac{\beta}{\alpha} $. There exists a constant $C_{p,\alpha}>0$ such that for every measurable set $A \subset X$,
\[
\mu(A)^{1-\frac{p\alpha}{\beta}}\le C_{p,\alpha} \mathbf{Cap}^\alpha_p (A).
\]
\end{corollary}

\begin{proof}
This is an immediate corollary of Theorem \ref{pol}.
\end{proof}

\subsection{Isoperimetric inequalities}\label{S:IsopIneq}

Let $E \subset X$ be a measurable set with finite measure. We will say that $E$ has a finite $\alpha$-perimeter if $\mathbf{1}_E \in \mathbf{B}^{1,\alpha}(X)$. In that case, we will denote
\[
P_\alpha(E)= \| \mathbf{1}_E \|_{1,\alpha}.
\]

\begin{proposition}
	\label{prop:iso}
Let $0 <\alpha < \beta$. For every $E\subset X$ with finite $\alpha$-perimeter
\begin{gather*}
\mu(E)^{\frac{\beta-\alpha}{\beta}} \le C_{\emph{iso}} P_\alpha (E),\end{gather*}\text{where}\begin{gather*}
C_{\emph{iso}} =
\frac{ C^{\frac{\alpha}{\beta}}
(\alpha+\beta)^{\frac{\alpha+\beta}{\beta}}}{2\beta\alpha^{\frac{\alpha}{\beta}}} . 
\end{gather*}
and $C$ is the constant from~\eqref{eq:subGauss-upper3 intro}  and~\eqref{eq:subGauss-upper3}.
\end{proposition}

\begin{proof}
We follow an argument originally due to M. Ledoux \cite{Ledoux}. Observe, 
from symmetry and conservativeness of the semigroup, that 
\begin{align*}
\Vert P_t \mathbf{1}_E -\mathbf 1_E  \Vert_{L^1(X,\mu)} =& \int_E (1-P_t \mathbf 1_E ) d\mu + \int_{E^c} P_t \mathbf 1_E\,  d\mu\\
 =& \int_E (1-P_t \mathbf 1_E ) d\mu + \int_E P_t \mathbf1_{E^c}\, d\mu\\
 =& 2 \left( \mu(E)- \int_E P_t \mathbf 1_E \, d\mu \right)
\end{align*}
However, our assumption~\eqref{eq:subGauss-upper3} on the heat kernel gives
\begin{equation*}
	\int_E P_t \mathbf 1_E \, d\mu
	= \int_E\int_E p_t(x,y) \,d\mu(x)\,d\mu(y)
	\leq Ct^{-\beta} \mu(E)^2
	\end{equation*}
Combining these equations with the pseudo-Poincar\'e estimate from Lemma~\ref{pseudo-Poincare} gives, for $t>0$,
\begin{equation*}
2\mu(E) \leq  \|P_t \mathbf 1_E - \mathbf 1_E\|_{L^1(X,\mu)}+2 \int_E P_t \mathbf 1_E \, d\mu
	\leq   t^\alpha \ P_\alpha (E) +2 C t^{-\beta} \mu(E)^2.
\end{equation*}
If $P_\alpha(E)=0$, then letting $t\to\infty$ we also see that $\mu(E)=0$, and so without loss of generality we may
assume that $P_\alpha(E)>0$, in which case we also have $\mu(E)>0$.
We minimize the right-hand side by choosing $t$ satisfying 
\[
t^{\alpha+\beta}= \frac{2C\beta}{\alpha} \frac{\mu(E)^2}{P_\alpha(E)}
\] 
and obtain, for this choice of $t$, that 
\[
2\mu(E)\le \left(1+\frac{\alpha}{\beta}\right) t^\alpha P_\alpha(E).
\]
Raising both sides of this equality to the power $\frac{\alpha+\beta}{\beta}$
concludes the proof.
\end{proof}

We note that in the limiting case $\alpha\to\beta$, the previous estimate for $C_{\mathrm{iso}}$ yields the following corollary, which is applicable on many fractal spaces, see \cite{ABCRST3,ABCRSTfr}:

\begin{corollary}
For every set  $E\subset X$ with finite $\beta$-perimeter and $\mu(E) >0$,
\[
P_\beta(E) \ge \frac{1}{2C}, 
\]
where $C$ is the constant from~\eqref{eq:subGauss-upper3}.
\end{corollary}

\subsection{Strong Sobolev inequality}

We now prove strong Sobolev inequalities. This requires an additional assumption on the space.

\begin{definition}\label{chaining}
We say that the Dirichlet space satisfies the property $(P_{p,\alpha})$ if there exists a constant $C>0$ such that for every 
$f \in \mathbf{B}^{p,\alpha}(X)$,
\[
\| f \|_{p,\alpha} \le C \liminf_{t \to 0} t^{-\alpha} \left( \int_X \int_X |f(x)-f(y) |^p p_t (x,y) d\mu(x) d\mu(y) \right)^{1/p}.
\]
\end{definition}

\begin{remark}
The property $(P_{p,\alpha})$ with $\alpha=1/2$ can be seen as a stronger form of the Proposition \ref{Critical bound Chapter 1}. As shown in \cite{ABCRST2,ABCRST3}, in many situations, the property $(P_{p,\alpha})$ is  satisfied, provided that the space $X$ satisfies a weak Bakry-\'Emery type non-negative curvature condition and that $\alpha$ is the $L^p$-Besov critical exponent of $X$ (see Section \ref{section critical} for the definition of Besov critical exponent). 

For instance,  $(P_{p,\alpha})$ is   satisfied for $p=1, \alpha=1/2$ for the standard Dirichlet form of $\mathbb{R}^n$. It is also satisfied for $p=1, \alpha=1/2$  for  the standard Dirichlet form of a complete Riemannian manifold with non-negative Ricci curvature. More generally, in the framework of \cite{ABCRST2},  property $(P_{p,\alpha})$  is satisfied when $p=1, \alpha=1/2$, see  Theorem 5.2 there. In the framework of Dirichlet spaces with sub-Gaussian heat kernel estimates, see \cite{ABCRST2},
 the property $(P_{1,\alpha})$ is satisfied for some $\alpha$ (the  $L^1$ Besov critical exponent) provided that a weak Bakry-\'Emery type estimate is satisfied. 
\end{remark}

Our main theorem is then the following:

\begin{theorem}\label{Sobolev}
Assume that the Dirichlet space satisfies the property $(P_{p,\alpha})$
and that $\beta$ is given in~\eqref{eq:subGauss-upper3}.
Let $0<\alpha <\beta $. Let $1 \le p < \frac{\beta}{\alpha} $. There exists a 
constant $C_{p,\alpha,\beta}>0$ such that for every $f \in \mathbf{B}^{p,\alpha}(X)$,
\[
\| f \|_{L^q(X,\mu)} \le C_{p,\alpha,\beta} \| f \|_{p,\alpha},
\]
where $q=\frac{p\beta}{ \beta-p \alpha}$.
\end{theorem}

Note that in the standard Euclidean setting of $\mathbb{R}^n$ the Sobolev embedding theorem holds as above with
$\beta=n$.
To show that the weak type inequality implies the desired Sobolev inequality, we will need another cutoff argument and the following lemma is needed.

\begin{lemma}
For $f \in \mathbf{B}^{p,\alpha}(X)$, $f \ge 0$, denote $f_k=(f-2^k)_+ \wedge 2^k$, $k \in \mathbb{Z}$. There exists a constant $C>0$ such that for every $f \in \ \mathbf{B}^{p,\alpha}(X)$,
\[
\left( \sum_{k \in \mathbb{Z}} \| f_k\|_{p,\alpha}^p \right)^{1/p} \le C \| f \|_{p,\alpha}.
\]
\end{lemma}

\begin{proof}
By a similar type of argument, as in the the proof of Lemma 7.1 in \cite{BCLS}, one has for some constant $C_p>0$,
\[
\sum_{k \in \mathbb{Z}} \int_X \int_X |f_k (x)-f_k(y)|^p p_t(x,y)d\mu(x) d\mu(y) \le C_p \int_X \int_X |f (x)-f(y)|^p p_t(x,y)d\mu(x) d\mu(y) .
\]
As a consequence of property $(P_{p,\alpha})$,
\[
\left( \sum_{k \in \mathbb{Z}} \| f_k\|_{p,\alpha}^p \right)^{1/p} \le C'_p \| f \|_{p,\alpha}.
\]
and the proof is complete.
\end{proof}

We can now conclude the proof of Theorem \ref{Sobolev}.

\begin{proof}[\textbf{Proof of Theorem \ref{Sobolev}}]
Let $f \in \mathbf{B}^{p,\alpha}(X)$. We can assume $f \ge 0$. As before, denote $f_k=(f-2^k)_+ \wedge 2^k$, $k \in \mathbb{Z}$. From Lemma \ref{pol} applied to $f_k$, we see that

\[
\sup_{s \ge 0} s \mu \left( \{ x \in X\, :\, | f_k(x) | \ge s \} \right)^{\frac{1}{q}} \le C_{p,\alpha} \| f_k \|_{p,\alpha}.
\]
In particular for $s=2^k$, we get
\[
2^k \mu \left( \{ x \in X\, :\, f (x) \ge 2^{k+1} \} \right)^{\frac{1}{q}} \le C_{p,\alpha} \| f_k \|_{p,\alpha}.
\]
Therefore,
\[
\sum_{k \in \mathbb{Z}} 2^{k q} \mu \left( \{ x \in X\, :\, f (x) \ge 2^{k+1} \} \right)
\le C_{p,\alpha}^q \sum_{k \in \mathbb{Z}} \| f_k \|_{p,\alpha}^q.
\]
Since $q \ge p$, one has 
$\sum_{k \in \mathbb{Z}} \| f_k \|_{p,\alpha}^q\le \left( \sum_{k \in \mathbb{Z}} \| f_k \|_{p,\alpha}^p \right)^{q/p}$. 
Thus, from the previous lemma
\[
\sum_{k \in \mathbb{Z}} 2^{k q} \mu \left( \{ x \in X\, :\, f (x) \ge 2^{k+1} \} \right)\le C_{p,\alpha}^q \|  f \|_{p,\alpha}^q.
\]
Finally, we observe that
\begin{align*}
\sum_{k \in \mathbb{Z}} 2^{k q} \mu \left( \{ x \in X\, :\, f (x) \ge 2^{k+1} \} \right) 
&\ge \frac{q}{2^{q+1}-2^q} \sum_{k \in \mathbb{Z}}\int_{2^{k+1}}^{2^{k+2}} s^{q-1} \mu \left( \{ x \in X\, :\, f (x) \ge s \} \right)ds \\
 &\ge \frac{1}{2^{q+1}-2^q} \| f \|_{L^q(X,\mu)}^q.
\end{align*}
The proof is thus complete.
\end{proof}

\subsection{Application}

The Sobolev embeddings studied in this section have many applications that will be studied in great details in the papers \cite{ABCRST2,ABCRST3,ABCRST4}. We just mention here that by combining Theorem \ref{sub gaussian intro} and Theorem \ref{polintro2}  one immediately obtains:

\begin{corollary}
Let $X$ be an Ahlfors $d_H$-regular space that satisfies sub-Gaussian  heat kernel estimates  as in Theorem \ref{sub gaussian intro}. Then, one has the following weak type Besov space embedding. Let $0<\delta < d_H $. Let $1 \le p < \frac{d_H}{\delta} $. There exists a constant $C_{p,\delta} >0$ such that for every $f \in \mathbf{B}^{p,\delta/d_W}(X) $,
\[
\sup_{s \ge 0} s \mu \left( \{ x \in X, | f(x) | \ge s \} \right)^{\frac{1}{q}} \le C_{p,\delta} \sup_{r>0} \frac{1}{r^{\delta+d_{H}/p}}
\biggl(\iint_{\Delta_r} 
|f(x)-f(y)|^{p}\,d\mu(x)\,d\mu(y)\biggr)^{1/p}
\]
where $q=\frac{p d_H}{ d_H -p \delta}$. Furthermore, for every $0<\delta <d_H $, there exists a constant $C_{\emph{iso},\delta}$ such that for every measurable $E \subset X$, $\mu(E) <+\infty$,
\begin{align}\label{isoperimetric intro}
\mu(E)^{\frac{d_H-\delta}{d_H}} \le C_{\emph{iso},\delta} \sup_{r>0} \frac{1}{r^{\delta+d_{H}}} (\mu \otimes \mu) \left\{ (x,y) \in E \times E^c\, :\, d(x,y)  < r\right\}.
\end{align}
\end{corollary}

\begin{remark}
 The number $\delta$ in the previous corollary plays the role of the upper codimension of the boundary of $E$. This will be further commented in \cite{ABCRST3}. 
\end{remark}

\begin{proof}
From the upper sub-Gaussian estimate, one has
\[
p_{t}(x,y)\leq C t^{-\beta}
\]
where $\beta=d_H/d_W$.
Let $0<\alpha <\beta $. Let $1 \le p < \frac{\beta}{\alpha} $. From Theorem \ref{polintro2}, there exists a constant $C_{p,\alpha} >0$ such that for every $f \in \mathbf{B}^{p,\alpha}(X) $,
\[
\sup_{s \ge 0}\, s\, \mu \left( \{ x \in X\, :\, | f(x) | \ge s \} \right)^{\frac{1}{q}} \le C_{p,\alpha} \| f \|_{p,\alpha},
\]
where $q=\frac{p\beta}{ \beta -p \alpha}$. However, from Theorem \ref{sub gaussian intro},
\[
\| f \|_{p,\alpha} \le C \sup_{r>0} \frac{1}{r^{\alpha d_W+d_{H}/p}}
 \biggl(\iint_{\Delta_r} 
 |f(x)-f(y)|^{p}\,d\mu(x)\,d\mu(y)\biggr)^{1/p}.
 \]
 The result follows then with $\delta=\alpha d_W$.
\end{proof}

\section{Cheeger constant and Gaussian isoperimetry}

While the previous section was devoted to Sobolev inequalities on Dirichlet spaces for which the semigroup satisfies ultracontractive estimates, the present section  is devoted to situations where the Dirichlet form satisfies a Poincar\'e inequality or a log-Sobolev inequality.

\subsection{Buser's type  inequality for the Cheeger constant of a Dirichlet space}
In the context of  a smooth compact  $n$-dimensional
Riemannian manifold with a normalized Riemannian measure $\mu$, Cheeger   introduced in \cite{Ch}   the following isoperimetric constant
\[
h=\inf \frac{\mathcal{H}^{n-1}(\partial A)}{\mu (A)},
\]
where $\mathcal{H}^{n-1}(\partial A)$ denotes the perimeter measure of $A$ and where the infimum runs over all open subsets $A$ with smooth boundary $\partial A$ such that $\mu(A)\le \frac{1}{2}$. Cheeger's constant can be used to bound from below the first non zero eigenvalue of the manifold. Indeed, it is proved in \cite{Ch} that
\[
\lambda_1 \ge \frac{h^2}{4}.
\]

Buser \cite{Bu} then proved that if the Riemannian Ricci curvature of the manifold is non-negative, then we actually have
\[
\lambda_1 \le C h^2,
\]
where $C$ is a universal constant depending only on the dimension. Buser's inequality was reproved by Ledoux \cite{Le} using heat semigroup techniques. Under proper assumptions, by using the tools we introduced in the present paper, Ledoux' technique can be essentially  reproduced in our general framework of Dirichlet spaces.

Let $(X,\mu,\mathcal{E},\mathcal{F})$ be a  Dirichlet space such that $\mu(X)=1$. Let $\{P_{t}\}_{t\in[0,\infty)}$ denote the semigroup associated with $(X,\mu,\mathcal{E},\mathcal{F})$.  The Dirichlet form $\mathcal{E}$ is said to satisfy a spectral gap inequality with spectral gap $\lambda_1$ if for every $f \in \mathcal{F}$,
\[
\int_X f^2 d\mu -\left( \int_X f d\mu \right)^2 \le \frac{1}{\lambda_1} \mathcal{E}(f,f).
\]
 We assume in this section that $\mathcal{E}$ satisfies a spectral gap inequality. For $\alpha \in (0,1]$, we define the $\alpha$-Cheeger's constant  of $X$ by
\[
h_\alpha=\inf \frac{\| \mathbf 1_E \|_{1,\alpha}}{\mu(E)},
\]
where the infimum runs over all measurable sets $E$ such that $\mu(E)\le \frac{1}{2}$ and $\mathbf 1_E \in \mathbf{B}^{1,\alpha}(X)$. We denote by $\lambda_1$ the spectral gap of $\mathcal{E}$.

\begin{theorem}
We have $h_\alpha \ge (1-e^{-1}) \lambda^\alpha_1$.
\end{theorem}

\begin{proof}
Let $A$ be a set with $P_\alpha (A) :=\| \mathbf 1_A \|_{1,\alpha}< +\infty$. As shown in the proof of Proposition \ref{prop:iso}, we have
\begin{align*}
\Vert \mathbf 1_A - P_t \mathbf 1_A \Vert_{L^1(X,\mu)} 
 =2 \left( \mu(A)- \Vert P_\frac{t}{2} (\mathbf 1_A) \Vert_{L^2(X,\mu)}^2 \right).
\end{align*} 
By the pseudo-Poincar\'e inequality in Lemma \ref{pseudo-Poincare}, 
\[
\| P_t \mathbf 1_A -\mathbf 1_A \|_{L^1(X,\mu)} \le t^\alpha P_\alpha(A).
\]
We deduce that
\[
\mu(A) \le \frac{1}{2} t^\alpha P_\alpha(A)+ \Vert P_\frac{t}{2} (\mathbf 1_A) \Vert_{L^2(X,\mu)}^2.
\]
Now, by the spectral theorem,
\[
\Vert P_\frac{t}{2} (\mathbf 1_A) \Vert_{L^2(X,\mu)}^2=\mu(A)^2 + \Vert P_\frac{t}{2} (\mathbf 1_A-\mu(A)) \Vert_{L^2(X,\mu)}^2 \le \mu(A)^2 +e^{-\lambda_1 t} \| \mathbf 1_A-\mu(A) \|_{L^2(X,\mu)}^2.
\]
This yields
\[
\mu(A) \le \frac{1}{2} t^\alpha P_\alpha(A)+ \mu(A)^2 +e^{-\lambda_1 t} \| \mathbf 1_A-\mu(A) \|_{L^2(X,\mu)}^2.
\]
Equivalently, one obtains
\[
\frac{1}{2} t^\alpha P_\alpha(A) \ge \mu(A) (1 -\mu(A))(1-e^{-\lambda_1 t}).
\]
Therefore,
\[
h_\alpha \ge \sup_{t >0} \left( \frac{1-e^{-\lambda_1 t}}{t^\alpha} \right),
\]
which completes the proof.
\end{proof}

As already noted in \cite{BK}, let us observe that it is known that the Cheeger lower bound on $\lambda_1$ may be obtained under further assumptions on the Dirichlet space $(X,d,\mathcal{E})$. Indeed, assume  that $\mathcal{E}$ is strictly local with a carr\'e du champ $\Gamma$,  that Lipschitz functions are in the domain of $\mathcal{E}$ and that $\sqrt{\Gamma(f)}$ is an upper gradient in the sense that for any Lipchitz function $f$,
\[
\sqrt{\Gamma(f)}(x) =\lim \sup_{d(x,y)\to 0} \frac{ |f(x)-f(y)|}{d(x,y)}.
\]
In that case, if $A$ is a closed subset of $X$, we define its Minkowski exterior boundary measure by
\[
\mu_+(A)=\lim \inf_{\varepsilon \to 0} \frac{1}{\varepsilon} \left( \mu(A_\varepsilon) -\mu(A) \right),
\]
where $A_\varepsilon=\{ x \in X, d(x,A) <\varepsilon \}$. We can then define a Cheeger's constant of $X$ by
\[
h_+=\inf \frac{\mu_+(E)}{\mu(E)},
\]
where the infimum runs over all closed sets $E$ such that $\mu(E)\le \frac{1}{2}$. Then, according to 
Theorem~8.5.2 in~\cite{BGL}, one has
\[
\lambda_1 \ge \frac{h^2_+}{4}.
\]

\subsection{Log-Sobolev and Gaussian isoperimetric inequalities}

Let $(X,\mu,\mathcal{E},\mathcal{F})$ be a  Dirichlet space such that $\mu(X)=1$. Let $\{P_{t}\}_{t\in[0,\infty)}$ denote the semigroup associated with $(X,\mu,\mathcal{E},\mathcal{F})$.  The Dirichlet form $\mathcal{E}$ is said to satisfy a log-Sobolev inequality with constant $\rho_0$ if for every $f \in \mathcal{F}$, $f \ge 0$,
\begin{equation}\label{log Sob}
\int_X f^2 \ln f^2 d\mu -\int_X f^2 d\mu \ln \int_X f^2 d\mu \le \frac{1}{\rho_0} \mathcal{E}(f,f).
\end{equation}

We assume in this section that $\mathcal{E}$ satisfies a log-Sobolev inequality  with constant $\rho_0$.  We define the Gaussian isoperimetric constant of $X$ by
\[
k=\inf \frac{\| \mathbf{1}_E \|_{1,1/2}}{\mu(E)\sqrt{-\ln \mu(E)}},
\]
where the infimum runs over all sets $E$ such that $\mu(E)\le \frac{1}{2}$ and $\mathbf{1}_E \in \mathbf{B}^{1,1/2}(X)$.  Following an argument of M. Ledoux \cite{Le}, we prove the following:

\begin{theorem} 
We have
\[
\rho_0 \le C_{\text{l}} k^2
\]
where $C_{\text{l}}$ is a numerical constant.
\end{theorem}

\begin{proof}
Let $A$ be a measurable set such that $P(A):=\| \mathbf 1_A \|_{1,1/2} <+\infty$. By the same computations as before we have
\[
\mu(A) \le \frac{1}{2} \sqrt{t} P(A)+ \Vert P_\frac{t}{2} (\mathbf 1_A) \Vert_{L^2(X,\mu)}^2.
\]
Now we can use the log-Sobolev constant to bound $\Vert P_\frac{t}{2} (\mathbf 1_A) \Vert_2^2$. Indeed, from Gross' theorem it is well known that the logarithmic Sobolev inequality 
\[
\int_X f^2 \ln f^2 d\mu -\int_X f^2 d\mu \ln \int_X f^2 d\mu \le \frac{1}{\rho_0} \mathcal{E}(f,f),
\]
is equivalent to the following hypercontractivity property of the semigroup
$$
\Vert P_t f \Vert_{L^q(X,\mu)} \leq \Vert f \Vert_{L^p(X,\mu)}
$$
 for all $f$ in $L^p(X,\mu)$ whenever $1<p<q<\infty$ and $e^{\rho_0 t}\geq \sqrt \frac{q-1}{p-1}$. Therefore, with $p(t)= 1+e^{-2\rho_0 t}<2$, we get
\begin{align*}
 \sqrt{t} P(A) \geq & 2 \left( \mu(A)- \mu(A)^\frac{2}{p(t)} \right)\\
 \ge & 2 \mu(A) \left(1- \mu(A)^\frac{1-e^{2-\rho_0 t}}{1+e^{-2\rho_0 t}} \right).
\end{align*}
By using then the computation page 956 in \cite{Le}, one deduces that if $A$ is a set which has a finite $P(A)$ and such that 
$0\leq \mu(A)\leq \frac{1}{2}$, then 
$$
P( A)\geq \widetilde{C} \sqrt{ \rho_0} \mu(A)\left(\ln \frac{1}{\mu(A)} \right)^\frac{1}{2},
$$
where $\widetilde{C}$ is a numerical constant.
\end{proof}

\bibliographystyle{plain}
 \bibliography{BV_Refs}

\

\ 

\noindent
P.A.R.: \url{paruiz@math.tamu.edu}  \\
Department of Mathematics,
Texas A\&M University,
College Station, TX 77843

\

\noindent
N.S.: \url{shanmun@uc.edu} \\
Department of Mathematical Sciences, 	
P.O. Box 210025,
University of Cincinnati, 	
Cincinnati, OH 45221-0025

\

\noindent
F.B.: \url{fabrice.baudoin@uconn.edu}
L.C.: \url{li.4.chen@uconn.edu}
L.R.: \url{rogers@math.uconn.edu}
A.T.: \url{teplyaev@uconn.edu}\\
Department of Mathematics,
University of Connecticut,
Storrs, CT 06269

\end{document}